\documentclass[12pt]{article}

\usepackage{amsmath}
\usepackage{amsfonts}
\usepackage{amsthm}
\usepackage{bbold}
\usepackage{bbm}
\usepackage{mathtools}
\usepackage{url}
\usepackage{tikz}
\usepackage{mathtools}
\mathtoolsset{showonlyrefs}
\usepackage{authblk}

\newtheorem{thm}{Theorem}[section]
\newtheorem*{thm*}{Theorem}
\newtheorem{prop}[thm]{Proposition}
\newtheorem*{prop*}{Proposition}
\newtheorem{cor}[thm]{Corollary}

\newtheorem*{thm:reversal}{Theorem \ref{reversalidentify}}

\newcommand{\RR}{\mathbb{R}}
\newcommand{\NN}{\mathbb{N}}
\newcommand{\ZZ}{\mathbb{Z}}
\newcommand{\PP}{\mathbb{P}}
\newcommand{\EE}{\mathbb{E}}
\newcommand{\QQ}{\mathbb{Q}}
\newcommand{\RRR}{\mathcal{R}}
\newcommand{\LLL}{\mathcal{L}}
\newcommand{\DDD}{\mathcal{D}}
\newcommand{\ZZZ}{\mathcal{Z}}
\newcommand{\FFF}{\mathcal{F}}

\newcommand{\Var}{\mathrm{Var}}

\providecommand{\keyword}[1]{\textbf{\textit{Keywords---}} #1}

\DeclarePairedDelimiter{\ceil}{\lceil}{\rceil}
\DeclarePairedDelimiter{\floor}{\lfloor}{\rfloor}
\DeclarePairedDelimiter{\norm}{\lVert}{\rVert}
\DeclarePairedDelimiter{\abs}{\lvert}{\rvert}

\title{Time-Reversal of Coalescing Diffusive Flows and Weak Convergence of Localized Disturbance Flows}
\author{James Bell}

\affil{Cambridge Centre for Analysis and The Alan Turing Institute, jbell@turing.ac.uk}

\begin{document}

\maketitle
\begin{abstract}
	We generalize the coalescing Brownian flow, also known as the Brownian web, considered as a weak flow to allow varying drift and diffusivity in the constituent diffusion processes and call these flows coalescing diffusive flows. We then identify the time-reversal of each coalescing diffusive flow and provide two distinct proofs of this identification. One of which is direct and the other proceeds by generalizing the concept of a localized disturbance flow to allow varying size and shape of disturbances, we show these new flows converge weakly under appropriate conditions to a coalescing diffusive flow and identify their time-reversals.
\end{abstract}

\keyword{Stochastic Flow, Distrubance Flow, Arratia Flow, Dual Flow, Time-Reversed Flow, Coalescing Flow}

\vspace{10pt}
\section{Introduction}
\vspace{10pt}

This paper is a contribution to the theory of stochastic flows in one dimension, specifically the study of inhomogeneous flows and their time-reversals. 

We provide two proofs of our main result which is Theorem \ref{reversalidentify} which says that the time-reversal of a coalescing diffusive flow with drift $b$ and diffusivity $a$ is (provided the spatial derivative $a'$ of $a$ is Lipschitz) given by a coalescing diffusive flow of drift $-b+\frac{a'}{2}$ and diffusivity $a$. Theorem \ref{flowconv} which establishes convergence of certain families of inhomogeneous disturbance flows to coalescing diffusive flows may also be of independent interest.

Coalescing Brownian motions were introduced by Arratia in 1979 \cite{A}. The object of study there consisted of a collection of coalescing Brownian motions starting from every point on the real line at the same time. T\'oth and Werner \cite{TW98} extended this to allow a Brownian motion to start from every point on the line at every time $t\in \RR$. Formally, this object is a family of random measurable functions $(\phi_{ts}:s\leq t \in \RR)$ satisfying the flow property
\begin{equation}
\phi_{ts}\circ \phi_{sr}=\phi_{tr}, \hspace{10mm} r\leq s \leq t
\end{equation}
and such that every finite collection of trajectories $(\phi_{ts}(x):t\geq s)$ performs coalescing Brownian motion. This is the approach taken in Arratia \cite{A}, T\'oth and Werner \cite{TW98}, Le Jan and Raimond \cite{LjR} and Tsirelson \cite{T}. A problem, however, with this approach is that the $\phi_{ts}$ cannot be chosen to be right-continuous, as the composition of two right-continuous functions is not necessarily right-continuous. 

An alternative approach that avoids this problem is given by Fontes et al. \cite{FINR} based on completing the set of trajectories to form a compact set of continuous paths, this completion can be done in multiple ways leading to multiple objects known as Brownian webs. Another way around the problem was introduced by Norris and Turner in \cite{NT}, based on the idea of considering pairs $\{\phi^- , \phi^+ \}$ of left and right continuous modifications of the Arratia flow. This setup does not store the information of the value of $\phi_{ts}$ at a jump, and as a result the flow property must be relaxed to a weak flow property (definition in Section \ref{CoalescingDiffusiveFlows}). The space of weak flows with the metric appearing in \cite{NT} provides a useful space for studying weak convergence, as it contains flows without continuous trajectories such as disturbance flows. This is the approach that this work builds on. Whilst Norris and Turner only deal with the case on the compact circle, this was extended to a Brownian web on $\mathbb{R}$ in the PhD thesis of Ellis \cite{E11}. A later paper by Berestycki et al. \cite{BGS13} provides another state-space and topology for the Brownian web, which was based on the quad crossings of Schramm and Smirnov and another topology is given in Greven et al \cite{GSW} based on marked metric measure spaces \cite{DGP}. A good overview of this work is given in Schertzer et al \cite{SSS15}.

Recently Riabov \cite{R18} has shown that coalescing stochastic flows can be realised as random dynamical systems. This approach avoids the need for relaxing to a weak flow property and constructs the time reversed (dual) flow explicitly as a part of the dynamical system.

The coalescing diffusive flow, $\phi$, consists of diffusion processes starting from each point in space-time, each with drift and diffusivity given by functions, $b$ and $a$ respectively, of space and time. They evolve independently until they collide, at which point they coalesce. We denote the distribution of this coalescing disturbance flow by $\mu_A$. The time reversal $\hat{\phi}$ of a flow $\phi$ is given by the inverse maps according to the following expression, where technicalities are being suppressed for brevity,
\begin{equation}
\hat\phi_I = \phi^{-1}_{-I}.
\end{equation}

Our main result in this paper is the following theorem, identifying the distribution of the time reversal of a coalescing stochastic disturbance flow. The distribution $\nu_A$ is of a coalescing stochastic flow with drift $b^\nu$ and diffusivity $a^\nu$.
\begin{thm:reversal}
  If $a$ has spatial derivative $a'$, and $a$, $b$ and $a'$, are uniformly bounded on compacts in time and $L$-Lipschitz in space then
  \begin{equation}
    \hat\mu_A=\nu_A:=\mu^{a^\nu,b^\nu}_A
  \end{equation}
  where	$ a^\nu(t,x)=a(-t,x)$, $b^\nu(t,x)=-b(-t,x)+a'(-t,x)/2$ and $\hat{\mu}_A$ is
  the time reversal of $\mu_A$.
\end{thm:reversal}

To intuitively  understand the presence of the $a'(-t,x)/2$ term consider the case $b\equiv 0$. If $a'(-t,x)$ and $\delta x$ are positive then the diffusivity at $x+\delta x$ is greater than at $x-\delta x$. This results in the flow from $x+\delta x$ over a small time increment from $t$ being more likely to drop below $x$ than the flow from $x-\delta x$ rising above $x$. This is equivalent to saying the time reversed flow is more likely to rise by $\delta x$ than fall by $\delta x$ over a short period of time. This manifests itself as a positive drift. Our direct proof will proceed partially along the lines of making this explanation exact.

A similar correction to the drift term can be found in work on smooth Brownian flows. For example a treatment in terms of infinitesimal generators can be found in \cite{K97}. This differs from our situation as the maps $\phi_I$ are required to be homeomorphisms and thus the paths can't be independent up to coalescence.

Since the original submission of this paper, Riabov has published a new paper \cite{R19}. There Riabov proves a version of the above result, for fixed diffusivity and with drift dependent on space but not time, within his random dynamical system framework. The paper provides a good framework for the consideration of stochastic flows. His existence proof in the previous paper \cite{R18} could straightforwardly be extended to this case. The calculations of the second method of proof in this paper could then be translated into that framework to show that the above theorem holds within that framework as well.

Konarovskyi \cite{K17} also studies a generalization of coalescing Brownian motions with varying diffusivity. In that work, the diffusions only start from time $t=0$ and the diffusivity of each is taken to be inversely proportional to the fraction of the diffusions that have coalesced to form it, rather than as a function of position and time, as in this work.

A disturbance flow, introduced in \cite{NT}, is a composition of independent random maps of the circle to itself. Unlike \cite{NT}, we do not require that our maps are identically distributed or that their distributions are invariant under conjugation by a rotation of the circle. For a pair of suitably smooth $a,b$, we consider limits where the maps $F$ are close to the identity, well localized and have mean of $F(x)-x$ close to $hb(x)$ and variance of $F(x)-x$ close to $ha(x)$ as $h\rightarrow 0$. We prove convergence of individual paths to diffusion processes and of the flow as a whole to the coalescing diffusive flow with diffusivity and drift given by $a$ and $b$. We also describe the time-reversal of the disturbance flows and use this to describe the time-reversal of a coalescing diffusive flow.

This paper is structured as follows. Section \ref{CountableDiffusions} proves existence and uniqueness of a simplified version of the coalescing diffusive flows, which consists of only countably many paths. Section \ref{CoalescingDiffusiveFlows} defines the metric spaces that our flows take values in, and proves existence and uniqueness of the coalescing diffusive flows (Theorem \ref{webexists}). Section \ref{TimeReversal} defines the time-reversal of a flow and provides the statement of our main result (Theorem \ref{reversalidentify}), which identifies the time-reversal of a coalescing diffusive flow. At this point the reader has the option of skipping straight to Section \ref{DirectProof} which will not require Sections \ref{DisturbanceFlows} or \ref{DisturbanceProof}. Section \ref{DisturbanceFlows} defines the notion of a disturbance flow, and shows convergence of paths from the flow to diffusions and of countable collections of paths to the simplified flow from Section \ref{CountableDiffusions}. Section \ref{DisturbanceProof} shows convergence of the disturbance flows to coalescing diffusive flows, identifies their time-reversals and uses this to provide a proof of Theorem \ref{reversalidentify}. Section \ref{DirectProof} provides an alternative proof of Theorem \ref{reversalidentify} that does not require the use of disturbance flows. It also contains as an intermediate weaker version (requiring more smoothness of $a$ and $b$) Theorem \ref{reversalLipschitz}.

The disturbance flow based proof of our main result generalizes the main proof in \cite{NT}; much of the notation is taken from there and some of the proofs are very similar. However, there are multiple places where new ideas are required to handle the generalization. While \cite{NT} allows the distribution of disturbances to be random only in that the location of the disturbance is chosen uniformly at random from around the circle, we allow the disturbances to vary in size and shape both randomly and with location in space and time, the shape and size is also allowed to vary a lot more as we take the limit to small disturbances than is allowed in \cite{NT}. The new ideas in the proofs are first evident in the proof of Theorem \ref{singleconv}, showing that individual trajectories of suitable disturbance flows converge weakly, where the proof of tightness requires bounds that hold despite the possibly varying drift and diffusivity. The time reversal results in Section \ref{DisturbanceProof} are generalizations of those in \cite{NT}. However, the statement of our main result Theorem \ref{reversalidentify} is not something that you would obviously expect, and the proof had to be modified substantially to deal with the more general disturbance flows.

The proof in Section \ref{DirectProof} is original in idea as well as in detail. While it is about the same length as the disturbance flow based proof, the weaker version of our main result Theorem \ref{reversalLipschitz} (which is identical except it assumes that $a$ and $b$ are Lipschitz in time as well as space) is proved with a substantially smaller amount of work (about 5 pages after the statement has been made rigorous rather than eighteen) and might suffice for future applications. In particular, it provides a short proof, without the use of disturbance flows, of the Brownian case which is Corollary 7.2 of \cite{NT}.

%\subsection{Acknowledgements}
%I am grateful to my supervisor James Norris for suggesting this topic and providing guidance and support as well as to fellow PhD student Vittoria Silvestri for providing moral support.

\vspace{10pt}
\section{Countable Collections of Coalescing Diffusions}
\label{CountableDiffusions}
\vspace{10pt}

In this section, we recall uniqueness in law for weak solutions of SDEs, then define a metric space, $D_E$, whose elements consist of countable collections of cadlag paths. Finally, using a martingale problem in the style of \cite{SV}, specifically those corresponding to a countable family of coalescing diffusion processes that are independent until collision; we identify certain elements of $D_E$.

Given functions $a:\RR^2 \rightarrow \RR$ and $b:\RR^2 \rightarrow \RR$ measurable, bounded uniformly on compacts in the first variable and $L$-Lipschitz in the second, with $a$ positive and bounded away from zero. Let $\sigma(t,x)=\sqrt{a(t,x)}$. Then the SDE
\begin{equation}
\label{SDE}
dX_t=b(t,X_t)dt+\sigma(t,X_t)dW_t
\end{equation}
has uniqueness in law for weak solutions \cite{RW}, i.e. given $e=(s,x)$ and a triple $(X,W)_{t\geq s}$, $(\Omega,\FFF,\PP)$, $(\FFF_t)_{t\geq s}$, such that

a) $(\Omega,\FFF,\PP)$ is a probability space with $(\FFF_t)_{t\geq s}$ as a complete, right-continuous filtration;

b) $X$ is adapted to $(\FFF_t)$, $X$ is continuous and $W$ is an $(\FFF_t)$-Brownian motion;

c) $X_s=x$;

d) almost surely, both $X$ and the quadratic variation of $X$ are bounded on each compact time interval;

e) almost surely
\begin{equation}
\label{SDEintegralform}
X_t=X_s+\int_s^t b(r,X_r)dr+\int_s^t \sigma(r,X_r)dW_r\hspace{5mm} \forall t\geq s,
\end{equation}

then the law of $X$ is determined by $a$,$b$ and $e$. Furthermore such solutions exist.

We will write this law as $\mu_e^{a,b}$, and say that $X$ is a diffusion process with drift $b$ and diffusivity $a$. Throughout we will assume that $a$ and $b$ have period $1$ in the second variable (as well as the properties above), and $X$ will be considered as a diffusion process on the circle $\RR/\ZZ$.

We will in several proofs use the notation
\begin{align}
b^*&:=\sup_{x\in [0,1], r\in I}|b(r,x)|,\\
a^*&:=\sup_{x\in [0,1], r\in I}a(r,x),\\
a_*&:=\inf_{x\in [0,1], r\in I}a(r,x),
\end{align}
where $I$ is an compact interval of time that contains all the times relevant to the given context. It will only be important that in any given context these numbers are finite and $a_*>0$.

Let $D_e=D_x([s,\infty),\RR)$ be the space of cadlag paths starting from $x$ at time $s$. Write $d_e$ for the Skorokhod metric on $D_e$.

Given a sequence $E=(e_k:k\in \NN)$ in $\RR^2$, set
\begin{equation}
D_E=\prod_{k=1}^{\infty}D_{e_k}
\end{equation}
and define a metric $d_E$ on $D_E$ by
\begin{equation}
d_E(z,z')=\sum_{k=1}^\infty 2^{-k} \left(d_{e_k}(z^k,z^{\prime k})\wedge1\right).
\end{equation}
Then $(D_E,d_E)$ is a complete separable metric space.

Write $e_k=(s_k,x_k)$ and denote by $(Z^k_t)_{t\geq s_k}$ the $k$th coordinate process on $D_E$, given by $Z^k_t(z)=z^k_t$. Consider the filtration $(\ZZZ_t)_{t\in \RR}$ on $D_E$, where $\ZZZ_t$ is the $\sigma$-algebra generated by $(Z^k_s : s_k<s\leq t\vee s_k, k \in \NN)$. Write $C_E$ for the (measurable) subset of $D_E$, where each coordinate path is continuous. Define on $C_E$
\begin{equation}
T^{jk}=\inf \{ t\geq s_j \vee s_k : Z^j_t - Z^k_t \in \ZZ \}.
\end{equation}
The $T^{jk}$ are the collision times of the paths considered in $\RR/\ZZ$. The following is a generalization of a reformulation in \cite{NT} of a result of Arratia in \cite{A}.

\begin{prop}
	\label{countableweb}
	Given $a,b$ measurable and bounded uniformly on compacts in time and $L$-Lipschitz in space as in \eqref{SDE}, there exists a unique Borel probability measure $\mu^{a,b}_E$ on $D_E$ under which, for all $j,k\in \NN$, the processes
	\begin{equation}
	\left( Z^k_t-\int_{s_k}^t b(s,Z^k_s) ds \right)_{t\geq s_k}
	\end{equation}
	and
	\begin{equation}
	\left(Z^k_tZ^j_t-\int_{s_j \vee s_k}^t \left(Z^k_sb(s,Z^j_s)+Z^j_sb(s,Z^k_s)\right)ds-\int_{T^{jk} \wedge t}^t a(s,Z^j_s) ds \right)_{t\geq s_j \vee s_k}
	\end{equation}
	are both continuous local martingales.
\end{prop}

We give the following proof sketch. For existence, one can take independent diffusion processes, with coefficients $a$ and $b$, from each of the given time-space starting points and then impose a rule of coalescence on collision, deleting the path of larger index. The law of the resulting process has the desired properties. On the other hand, given a probability measure such as described in the proposition, on some larger probability space, one can use a supply of independent Brownian motions to build diffusions continuing each of the paths deleted at each collision. Then, the martingale problem characterization of diffusion processes given in \cite{SV}, can be used to see that one has recovered the set-up used for existence. This gives uniqueness.

\vspace{10pt}
\section[Existence and Uniqueness of Coalescing Diffusive Flows]{Existence and Uniqueness of Coalescing\\ Diffusive Flows}
\label{CoalescingDiffusiveFlows}
\vspace{10pt}

We now introduce the space of continuous weak flows $C^\circ (\RR,\DDD)$ and the space of cadlag weak flows $D^\circ (\RR,\DDD)$, both introduced in \cite{NT}. We will then identify certain elements of $C^\circ (\RR,\DDD)$ as coalescing diffusive flows, again using a martingale problem. The space $C^\circ (\RR,\DDD)$ is sufficient for stating our main result and understanding the proof that doesn't use disturbance flows. However, we will need $D^\circ (\RR,\DDD)$ to deal with the fact that the disturbance flows are not continuous in time. The following explanation of notation follows \cite{NT} very closely, and all the claims made in italics are proved in \cite{NT}.

We consider non-decreasing, right-continuous functions $f^+:\RR \rightarrow \RR$ with the degree 1 property
\begin{equation}
f^+(x+n) = f^+(x)+n,\hspace{10mm} x\in\RR,\hspace{10mm} n\in\ZZ.
\end{equation}
Let us denote the set of such functions by $\RRR$ and the set of analogous left-continuous functions by $\LLL$. Each $f^+ \in \RRR$ has a left-continuous modification given by $f^-(x)=\lim_{y\uparrow x}f^+(y)$. Let $\DDD$ denote the set of corresponding pairs $f=\{f^-,f^+\}$. We will write $f$ in place of $f^{\pm}$ when the choice is irrelevant for the purpose at hand, especially in the case when $f^+=f^-$, i.e. $f^+$ is continuous.

Firstly, we define a metric on $\DDD$. Associate to each function $f$ a function $f^\times$ given by $f^\times (t)=t-x$, where $x\in \RR$ is the unique value such that
\begin{equation}
\frac{x+f^-(x)}{2}\leq t\leq \frac{x+f^+(x)}{2}
\end{equation}
as shown in Figure \ref{CoordChange}.
\emph{We can define a complete locally compact metric $(\DDD,d_\DDD)$ by}
\begin{equation}
d_\DDD(f,g)=\sup_{t\in[0,1)} \abs{f^\times(t)-g^\times(t)}.
\end{equation}
\begin{figure}
	\begin{tikzpicture}[scale=0.51]
	\draw [<->] (11,0) -- (0,0) -- (0,11);
	\draw [<->] (-11,0) -- (0,0) -- (0,-11);
	\node [above] at (0,11) {$f(x)$};
	\node [below] at (11,0) {$x$};
	\draw (-10,-10) -- (-8,-9) -- (-7,-5) -- (-3,-4) -- (-2,0) -- (4,3) -- (7,6) -- (7,8) -- (10,10);
	\draw [dashed] (-10,-10)--(10,10) (-10,10) -- (10,-10);
	\node [above] at (-10,10) {$f^\times (t)$};
	\node [right] at (10,10) {$t$};
	\end{tikzpicture}
	\caption{\label{CoordChange}The graph of $f^\times$ can be formed from the graph of $f$ by rotating the axes by $\frac{\pi}{4}$ and scaling both axes up by $\sqrt{2}$. Note that where a jump in $f$ occurs we must add a straight line between the upper and lower limits in order to give the graph of $f^\times$ there.}
\end{figure}
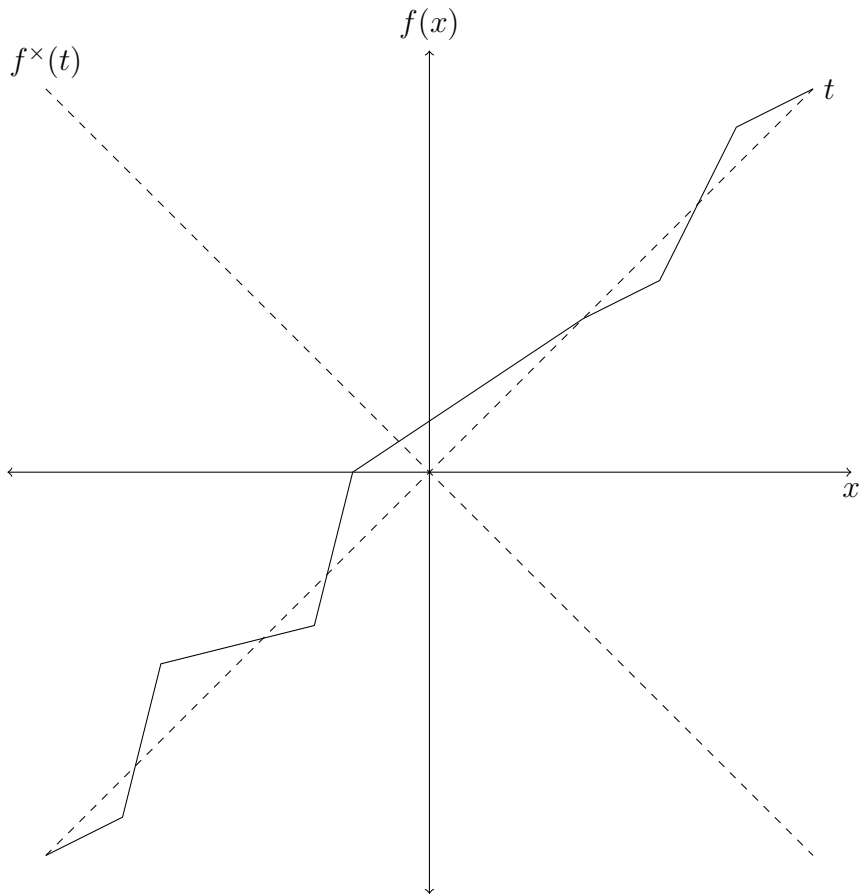

Consider $\phi = (\phi_I : I\subseteq \RR)$, with $\phi_I \in \DDD$ and $I$ ranging over all non-empty bounded intervals. We say that $\phi$ is a \emph{weak flow} if given $I$ a disjoint union of intervals $I_1$ and $I_2$, with $\sup I_1=\inf I_2$,
\begin{equation}
\phi^-_{I_2}\circ \phi^-_{I_1} \leq \phi^-_I \leq \phi^+_I \leq \phi^+_{I_2} \circ \phi^+_{I_1}.
\end{equation}
$\phi$ is said to be \emph{cadlag} if for all $t\in \RR$,
\begin{equation}
\phi_{(s,t)} \rightarrow \textrm{id} \hspace{5mm} \textrm{as} \hspace{5mm} s \uparrow t, \hspace{10mm} \phi_{(t,u)} \rightarrow \textrm{id} \hspace{5mm} \textrm{as} \hspace{5mm} u \downarrow t.
\end{equation}
Here, the convergence of functions is with respect to the metric of $\DDD$ (also note that this definition is left-right symmetric, we call it cadlag to match previous work).

$D^\circ (\RR,\DDD)$ is the set of cadlag weak flows. We set $\phi_\emptyset = \textrm{id}$.
Given $\{I_n : n\in \NN\}$ and $I$ bounded intervals, write $I_n\rightarrow I$ if
\begin{equation}
I=\bigcup_n \bigcap_{m\geq n} I_m = \bigcap_n \bigcup_{m\geq n} I_m.
\end{equation}
\emph{For every $\phi \in D^\circ(\RR,\DDD)$, we have}
\begin{equation}
\phi_{I_n}\rightarrow \phi_I \textrm{ whenever } I_n\rightarrow I.
\end{equation}

If $\phi \in D^\circ(\RR,\DDD)$ satisfies $\phi_{\{t\}}=\textrm{id}$ for all $t \in \RR$ then we have that $\phi_{(s,t)}=\phi_{(s,t]}=\phi_{[s,t)}=\phi_{[s,t]}$ for all $s<t$. Denoting these all by $\phi_{ts}$ we define $C^\circ (\RR,\DDD)$ to be the set of all such $(\phi_{ts}: s,t \in \RR, s<t)$.
For $\phi,\psi \in C^\circ(\RR,\DDD)$ and $n\geq1$, define
\begin{equation}
d^{(n)}_{C}(\phi,\psi)= \sup_{s,t\in(-n,n),s<t}d_\DDD(\phi_{ts},\psi_{ts})
\end{equation}
and then let
\begin{equation}
d_C(\phi,\psi)= \sum_{n=1}^{\infty} 2^{-n} \left( d^{(n)}_C(\phi,\psi)\wedge 1 \right).
\end{equation}
Under this metric $C^\circ(\RR,\DDD)$ \emph{is complete and separable}.

In the interests of defining a metric on $D^\circ(\RR,\DDD)$, for $\lambda$ an increasing homeomorphism of $\RR$ we define
\begin{equation}
\gamma(\lambda)=\sup_{t \in \RR} \abs{\lambda (t)-t}\vee \sup_{s,t\in\RR ,s<t} \left\lvert \log\left(\frac{\lambda(t)-\lambda(s)}{t-s}\right) \right\rvert,
\end{equation}
and let $\chi_n$ be the cut-off function given by
\begin{equation}
\chi_n(I)=0\vee (n+1-R)\wedge 1, \hspace{5mm} R=\sup I \vee (-\inf I).
\end{equation}
We can now define for $\phi,\psi \in D^\circ(\RR,\DDD)$ and $n\geq1$,
\begin{equation}
d^{(n)}_{D}(\phi,\psi)=\inf_\lambda \left\lbrace  \gamma(\lambda) \vee \sup_{I\subseteq \RR} \norm{\chi_n(I)\phi^\times_I - \chi_n(\lambda(I))\psi^\times_{\lambda(I)} }_\infty \right\rbrace
\end{equation}
where the infimum is taken over the set of increasing homeomorphisms $\lambda$ of $\RR$. Then define
\begin{equation}
d_D(\phi,\psi)= \sum_{n=1}^{\infty} 2^{-n} \left( d^{(n)}_D(\phi,\psi)\wedge 1 \right) .
\end{equation}
Then $(D^\circ(\RR,\DDD),d_D)$ is a \emph{complete and separable} metric space. Moreover \emph{$d_C$ and $d_D$ generate the same topology on $C^\circ (\RR,\DDD)$}. \emph{For the metric $d_D$, all bounded intervals $I$ and all $x \in \RR$, the} evaluation map
\begin{equation}
\phi \mapsto \phi^+_I (x) : D^\circ(\RR,\DDD) \rightarrow \RR
\end{equation}
is \emph{Borel measurable}. \emph{Moreover the Borel $\sigma-algebra$ on $D^\circ(\RR,\DDD)$ is generated by the set of all such evaluation maps with $I=(s,t]$ and $s,t$ and $x$ rational}.

\emph{For $e=(s,x) \in \RR$ and $\phi \in D^\circ(\RR,\DDD)$, the maps}
\begin{equation}
t \mapsto \phi^\pm_{(s,t]}(x) : [s,\infty) \rightarrow \RR
\end{equation}
\emph{are cadlag}. Hence we can define $Z^e = Z^{e,+}$ and $Z^{e,-}$, as maps from $D^\circ(\RR,\DDD)$ to $D_e$, by setting
\begin{equation}
Z^{e,\pm}(\phi)=(\phi^\pm_{(s,t]}(x):t\geq s).
\end{equation}
\emph{The maps, $t\rightarrow Z^{e,\pm}_t(\phi)$ are continuous when} $\phi \in C^\circ(\RR,\DDD)$.

Finally, define a $\sigma$-algebra $\FFF$ and a filtration $(\FFF_t)_{t\in \RR}$ on $C^\circ(\RR,\DDD)$ by
\begin{equation}
\FFF=\sigma(Z^{(s,x)}_t:{(s,x)}\in \RR^2, t\geq s)
\end{equation}
and
\begin{equation}
\FFF_t=\sigma(Z^{(s,x)}_r: {(s,x)}\in \RR^2, r\in (-\infty,t] \cap [s,\infty)).
\end{equation}
Then $\FFF_t$ \emph{is generated by the random variables} $Z^{(s,x)}_r$ with ${(s,x)}\in \QQ^2$ and $r \in (-\infty,t] \cap [s,\infty)$, and \emph{$\FFF$ is the Borel $\sigma$-algebra of the metric $d_C$}.

The following theorem states the existence of coalescing diffusive flows. The proof is identical to that of the less general result Theorem 3.1 in \cite{NT} and so is omitted. Generalizing the argument requires generalized versions of results from \cite{NT}, which we give as Proposition \ref{countableweb} and Proposition \ref{technicalproposition}.

Analogously to $T^{jk}$, if $e=(s,x)$ and $e'=(s',x')$ we define
\begin{equation}
T^{ee'}=\inf \{ t\geq s \vee s' : Z^e_t - Z^{e'}_t \in \ZZ \}.
\end{equation}

\begin{thm}
	\label{webexists}
	Given $a,b$ as before, there exists a unique Borel probability measure $\mu^{a,b}_A$ on $C^\circ(\RR,\DDD)$ under which, for all $e=(s,x),e'=(s',x')\in\RR^2$, the processes
	\begin{equation}
	\label{processone}
	\left( Z^e_t-\int_{s}^t b(r,Z^e_r) dr \right)_{t\geq s}
	\end{equation}
	and
	\begin{equation}
	\label{processtwo}
	\left(Z^e_tZ^{e'}_t-\int_{s \vee s'}^t\left( Z^{e}_rb(r,Z^{e'}_r)+Z^{e'}_rb(r,Z^e_r)\right)dr-\int_{T^{ee'} \wedge t}^t a(r,Z^{e'}_r) dr \right)_{t\geq s \vee s'}
	\end{equation}
	are continuous local martingales with respect to $(\FFF_t)_{t\in\RR}$. Moreover, for all $e\in\RR^2$ we have $\mu^{a,b}_A$-almost surely $Z^{e,+}=Z^{e,-}$.
\end{thm}

We will often write $\mu_A$ instead of $\mu^{a,b}_A$ in order to simplify notation.

\vspace{10pt}
\section{Time Reversal}
\label{TimeReversal}
\vspace{10pt}

In this section we quote some definitions and observations from \cite{NT} and then state our main theorem. For $f^+\in \RRR$ and $f^-\in \LLL$, define the \emph{left-continuous inverse} from $\RRR$ to $\LLL$ and the inverse operation \emph{right-continuous inverse} respectively as follows
\begin{equation}
(f^+)^{-1}(y)=\sup\{x\in \RR : f^+(x)<y\},
\end{equation}
\begin{equation}
(f^-)^{-1}(y)=\inf\{x\in \RR : f^-(x)>y\}.
\end{equation}
Note that these operations are distributive over concatenation. The \emph{inverse} of $f \in \mathcal{D}$ is given by
\begin{equation}
f^{-1}=\{(f^+)^{-1},(f^-)^{-1}\}\in \DDD.
\end{equation}
The \emph{time-reversal} $\hat\phi$ of a flow $\phi$ is given by
\begin{equation}
\hat\phi_I = \phi^{-1}_{-I}.
\end{equation}
The time-reversal map is a well defined isometry of both $D^\circ(\RR,\DDD)$ and $C^\circ(\RR,\DDD)$.

As before, let $a$ and $b$ be the diffusivity and drift of a diffusive flow with law $\mu_A$. We require that $a$ and $b$ satisfy the smoothness requirements of Section \ref{CountableDiffusions} and further require that $a$ is differentiable with respect to $x$ with derivative $a'(t,x)$, which is $L$-Lipschitz in $x$ and measurable and bounded uniformly on compacts in $t$.
Let $ a^\nu(t,x)=a(-t,x)$ and $b^\nu(t,x)=-b(t,-x)+a'(-t,x)/2$, let $\nu_A=\mu^{a^\nu,b^\nu}_A$,  i.e. let it be the law of a disturbance flow with drift and diffusivity given by $b^\nu$ and $a^\nu$. Finally, write $\hat\mu_A$ for the image measure of $\mu_A$ under time-reversal.

\begin{thm}
	\label{reversalidentify}
	The time-reversal of the diffusive flow $\mu_A$ is a diffusive flow with the new parameters given in the previous paragraph, i.e.
	\begin{equation}
	\hat\mu_A=\nu_A.
	\end{equation}
\end{thm}

We will provide two proofs of this theorem: one in Section \ref{DisturbanceProof} which depends on Section, \ref{DisturbanceFlows} and one in Section \ref{DirectProof} which does not depend on Sections \ref{DisturbanceFlows} or \ref{DisturbanceProof}.

\vspace{10pt}
\section{Disturbance Flows from Countably Many Points on a Circle}
\label{DisturbanceFlows}
\vspace{10pt}

This section lays the ground work for Section \ref{DisturbanceProof}. The reader may skip to Section \ref{DirectProof} at this point if they only wish to read the direct proof.

We start this section by defining the notion of a disturbance flow on the circle. This is based on a notion of disturbance flow which was given in \cite{NT}, but is more general, so as to allow for our disturbance flows to have drift and varying diffusivity. We will then proceed to state and prove two propositions and deduce a theorem. The propositions are as follows: firstly, under appropriate conditions a sequence of single paths from disturbance flows can converge to a diffusion process; and secondly, a sequence of countable families of paths from disturbance flows can converge to a countable family of coalescing diffusions. Combining these propositions with a result from \cite{NT}, we conclude that disturbance flows can converge to coalescing diffusive flows.

We specify a disturbance flow by a family of probability distributions on $\DDD$ written
\begin{equation}
\eta=\{\eta_{h,t}:h>0,t\in \RR\}.
\end{equation}
The parameters of the family are $h>0$, which corresponds to the size of the disturbance (the limit for our convergence later will be taking $h$ to $0$ while making disturbances more frequent) and time $t$, which allows our flow to be inhomogeneous in time. We require that $\eta$ be measurable as a function of $t$.

Given $f_1,f_2 \in \DDD$, define $f_2\circ f_1 := \{f_2^-\circ f_1^-,f_2^+\circ f_1^+\}$. This is not in general an element of $\DDD$, however, so long as $f_1$ sends no interval of positive length to a point of discontinuity of $f_2$, we will have $f_2\circ f_1\in \DDD$. To avoid this issue, we will only consider families of probability distributions on $\DDD$ such that, if $F_{h,t}\sim \eta_{h,t}$, then
\begin{equation}
\label{D*def}
F_{h,t}^+(x)=F_{h,t}^-(x) \textrm{ a.s. } \forall x,t \in \RR \textrm{ and } h\in \RR^+.
\end{equation}
Where $\RR^+=\{h \in \RR : h>0\}$.
We denote the set of such families by $\DDD^*$, and assume from here on that $\eta=\{\eta_{h,t}: h>0, t \in \RR \}\in\DDD^*$.  Let $N$ be a Poisson random measure on $\RR$ of intensity $h^{-1}$ and set
\begin{equation}
N_t=\begin{cases} N(0,t], &t\geq 0 \\
-N(t,0], &t<0. \end{cases}
\end{equation}
Let
\begin{equation}
t_n=\inf\{t : N_t\geq n\}
\end{equation}
and $\{F_{h,t_n}: h>0,n\in\NN\}$ be independent random variables with $F_{h,t_n}\sim \eta_{h,t_n}$. We will sometimes write $F_n$ for $F_{h,t_n}$.

We extend the inverse functions of Section \ref{TimeReversal} to families of probability distributions $F \in \DDD^*$ by setting
\begin{equation}
(F^{-1})_{h,t}(y)=(F_{h,t})^{-1}(y)
\end{equation}
where the inverse on the right hand side is being taken with respect to the $x$ argument (as opposed to the implicit $\omega$ argument). Also let $\tilde{F}_{h,t}(x)=F_{h,t}(x)-x$.

Then, for any interval $I$, define
\begin{equation}
\Phi_{I}(x)=x+\int_I \tilde{F}_{h,r}(\phi_{I\cap (-\infty,r)})N(dr).
\end{equation}
Write $\Phi$ for the family of maps $\Phi_I$ where $I$ ranges over all bounded intervals in $\RR$. We call $\Phi$ the \emph{Poisson disturbance flow} or just the \emph{disturbance flow} and write $\mu^\eta_A$ for the distribution of $\Phi$ in $D^\circ(\RR,\DDD)$.

Fixing $e=(s,x) \in \RR^2$ we define 2 processes $X^{e,\pm}_t$ by setting $X^{e,\pm}_t=\Phi^\pm_{(s,t]}(x)$ for $t\geq s$. Because $\Phi\in\DDD$ a.s. we have a.s. that for all $ t\in \QQ_{\geq s}$
\begin{equation}
X^{e,-}_t=X^{e,+}_t
\end{equation}
and thus by right continuity of $X^{e,\pm}$ we have a.s. that for all $ t\geq s$,
\begin{equation}
X^{e,-}_t=X^{e,+}_t.
\end{equation}
Thus, we drop the $\pm$ and write simply $X^e$. Write $\mu^\eta_e$ for the distribution of $X^e$ on the Skorokhod space $D_e$. Similarly, for $E=(e_k\in \RR^2:k\in\NN)$, $(X^{e_k}: k\in \NN)$ is a random variable in $D_E$, and we write $\mu^\eta_E$ for its distribution on $D_E$.

Given a family $\eta \in \DDD^*$, and coefficients $a$ and $b$ as in Section \ref{CountableDiffusions}, we define the functions
\begin{equation}
b_h(t,x)=\frac{1}{h}\EE(\tilde{F}_{h,t}(x))
\end{equation}
\begin{equation}
a_h(t,x)=\frac{1}{h}\EE( \tilde{F}_{h,t}(x)^2)
\end{equation}
\begin{equation}
M_h=\sup_{x\in [0,1], |t| \leq T, \omega \in \Omega } \abs{\tilde{F}_{h,t}(x)}
\end{equation}
\begin{equation}
B_h=\sup_{x\in [0,1], |t| \leq T} \abs{b_h-b}
\end{equation}
\begin{equation}
A_h=\sup_{x\in [0,1], |t| \leq T} \abs{a_h-a}.
\end{equation}
The following three conditions will be important for the next proposition and consequently for the rest of the results:
\begin{align}
\label{convbh}
\lim_{h\searrow 0} B_h=0
& \hspace{10mm}\forall T \in \RR^+ \\
\label{convah}
\lim_{h\searrow 0} A_h=0
& \hspace{10mm}\forall T \in \RR^+ \\
\label{jumpbound}
\lim_{h\searrow 0} M_h=0
& \hspace{10mm}\forall T \in \RR^+ .
\end{align}
\begin{prop}
	\label{singleconv}
	Suppose $a$ and $b$ are functions as specified for Equation \eqref{SDE}, and that $\eta$ is such that conditions \eqref{convbh}, \eqref{convah} and \eqref{jumpbound} hold. Then we have $\mu_e^\eta \rightarrow \mu_e^{a,b}$ weakly on $D_e$, as $h \rightarrow 0$.
\end{prop}

\begin{proof}
	Let $(X^n)_{n\in \NN}$ be a sequence of processes distributed according to $\mu_e^\eta$ with $h\to 0$ as $n\to \infty$. By the definition of the Skorokhod metric, it suffices to show that for any $T>s$, the restrictions of $X^n$ to $[s,T]$ converge weakly to a solution of the SDE on $[s,T]$. For the remainder of this proof we consider $X^n$ to be restricted to $[s,T]$. We then take $e=(s,x)=(0,0)$ and $T=1$, without loss of generality.
	
	Firstly, we shall calculate (up to an error that is small for small $|t-s|$) two expected values (defined in terms of $s,t\in \RR$). We shall then prove a characterization of tightness of the sequence. This will require us to use these calculations to show that the process can't vary too much on a given interval, then deduce the existence of a subsequential limit of each subsequence by Prokhorov's theorem. Finally we will identify the distribution of every subsequential limit as a weak solution of Equation \eqref{SDE}, using again the 2 expectation calculations. Then we will conclude the proof using the uniqueness in law for such solutions .
	
	Let $\FFF^n_t$ be the completion of the filtration generated by $X^n$. For $1 \geq t\geq s \geq 0$ we have
	\begin{align}
	&\EE(X^n_t-X^n_s| \FFF^n_s) \\
	=&e^{-\frac{t-s}{h}}\frac{t-s}{h}\EE(\tilde{F}_{N_s+1}(X^n_s)|t_{N_s+1}\leq t<t_{N_s+2}, \FFF^n_s) + E_1 \\
	=&\int_s^t b_h(r,X^n_s) dr + E_1 + E_2 \\
	=&\EE \left( \int_s^t b_h(r,X^n_r) dr \right)+ E_1 + E_2 + E_3.
	\end{align}
	Where the approximation errors $E_i$ can be bounded as follows. Note that in the above calculation $h$ is held constant so can be thought of as such for these bounds. Let $G_{s,k}=\tilde{F}_{N_s+k}\circ ... \circ \tilde{F}_{N_s+1}$
	\begin{align}
          \abs{E_1}=&\left| \sum_{k\geq 2} \exp\left(-\frac{t-s}{h}\right) \frac{(t-s)^k}{k!h^k}\EE(G_{s,k}(X^n_s)|t_{N_s+k}\leq t<t_{N_s+k+1}) \right| \\
	\leq& \exp\left(-\frac{t-s}{h}\right) \sum_{k\geq 2} \frac{(t-s)^k}{k!h^k}kM_h \\
	=& \exp\left(-\frac{t-s}{h}\right) \frac{t-s}{h} M_h \sum_{k\geq 1} \frac{(t-s)^k}{k!h^k} \\
	=& M_h \frac{t-s}{h} \exp\left(-\frac{t-s}{h}\right) \left(\exp\left(\frac{t-s}{h}\right)-1\right) \\
	=& O\left((t-s)^2 \right).
	\end{align}
        Note that $\int_s^t b_h(r,X^n_s) dr=\frac{t-s}{h} \EE(\tilde{F}_{N_s+1}(X^n_s)|t_{N_s+1}\leq t<t_{N_s+2})$ to understand $E_2$, which can be bounded as follows,
	\begin{align}
	\abs{E_2}&=\left(1-\exp\left(-\frac{t-s}{h}\right)\right) \frac{t-s}{h} \abs{\EE(\tilde{F}_{N_s+1}(X^n_s)|t_{N_s+1}\leq t<t_{N_s+2},\FFF^n_s)} \\
	&\leq \left(1-\exp\left(-\frac{t-s}{h}\right)\right) \frac{t-s}{h} M_h \\
                 &= O\left((t-s)^2 \right).
        \end{align}
        Then finally,
        \begin{align}
	\abs{E_3}&= \EE \left( \left|\int_{s}^{t} \left(b_h(r,X^n_s)-b_h(r,X^n_r)\right)dr \right|\middle| \FFF^n_s \right)\\
	&\leq (t-s)(2b^*+2B_h)\PP(t_{N_s+1}\leq t) \\
	&=O\left((t-s)^2 \right).
	\end{align}
	
	Breaking the interval $(s,t]$ into a large number of small intervals and taking the limit as the interval sizes go to $0$, we have that,
	\begin{align}
	\label{firstexpectationbound}
	\EE(X^n_t-X^n_s\mid \FFF^n_s) &= \EE\left( \int_s^t b_h(r,X^n_r) dr \bigg| \FFF^n_s \right).
	\end{align}
	Similarly
	\begin{align}
	\label{secondexpectationbound}
	\EE((X^n_t-X^n_s)^2\mid \FFF^n_s) &= \EE\left( \int_s^t a_h(r,X^n_r) dr\bigg| \FFF^n_s \right).
	\end{align}
	
	The characterization of tightness that we shall use is given in Billingsley 1968 \cite{B} Theorem 15.3, it says that tightness is equivalent to the following two conditions holding.
	\begin{enumerate}
		\item For all $\epsilon>0$ there exists a $K$ such that
		\begin{equation}
		\PP\left(\sup_t \abs{X^n_t}\geq K\right)\leq \epsilon, \hspace{10mm} \forall n\geq 1.
		\end{equation}
		\item 	Taking
		\begin{equation}
		w''_{X^n}(\delta)=\sup_{\substack{t_1\leq t \leq t_2\\ t_2-t_1\leq \delta}} \min\{\abs{X^n(t)-X^n(t_1)},\abs{X^n(t)-X^n(t_2)}\}
		\end{equation}
		and
		\begin{equation}
		w_{X^n}(I)=\sup_{s,t\in I}\abs{X^n_s-X^n_t}
		\end{equation}
		for all $\epsilon>0$ there exists $\delta \in (0,1)$ and $N\in \NN$ such that
		\begin{equation}
		\PP(w''_{X^n}(\delta)\geq \epsilon)\leq \epsilon, \hspace{10mm} \forall n\geq N
		\end{equation}
		and
		\begin{equation}
		\PP(w_{X^n}[0,\delta)\geq \epsilon)\leq \epsilon, \hspace{10mm} \forall n\geq N
		\end{equation}
		and
		\begin{equation}
		\PP(w_{X^n}(1-\delta,1]\geq \epsilon)\leq \epsilon, \hspace{10mm} \forall n\geq N.
		\end{equation}
	\end{enumerate}
	
	Note that $B_h$ and $A_h$ going to $0$ as $n\to \infty$ means that $|b_h|$ and $|a_h|$ are bounded uniformly in $n$, $x$ and $t\in [0,1]$. We call the bounds $B$ and $A$ respectively.
	
	The first condition can be shown as follows, where $T_K$ is the first time $t$ such that $X^n_t\geq K$.
	\begin{align}
          &\PP\left(\sup_{t\leq 1}X^n_t \geq K\right) \\
          =& \PP\left(T_K \leq 1\right) \\
          \leq& \PP\left( X^n_1 \geq \frac{K}{2}\right) +\PP\left( T_K \leq 1, X^n_1 \leq \frac{K}{2} \right) \\
          \leq& \PP\left( X^n_1 \geq \frac{K}{2} \right) + \EE \left( \PP \left( X^n_1-X^n_{T_k\wedge 1} \leq -\frac{K}{2} \bigg| \FFF_{T_K\wedge 1} \right) \right) \\
          \leq& \frac{2A}{(\frac{K}{2}-B)^2}
	\end{align}
	where in the final inequality we have used Chebyshev's inequality. This bound goes to $0$ as $K\to \infty$ uniformly in h. Combining with a corresponding bound for $\inf X^n_t$ gives the first condition.
	
	Note that for the second condition, it suffices to show the following stronger statement, where $\mathbb{I}_\delta$ is the set of subintervals of $[0,1]$ of length $\delta$.
	
	For all $\epsilon > 0$ there exists $\delta \in (0,1)$ and $N \in \NN$ such that
	\begin{equation}
	\label{continuitycondition}
	\PP(\exists I\in \mathbb{I}_\delta \textrm{ such that } w_{X^n}(I)\geq \epsilon) \leq \epsilon \hspace{10mm} \forall n>N
	\end{equation}
	which in turn is implied by the following, where $\mathbb{I}'_\delta$ is the set of intervals of length $\delta$ with endpoints that are multiples of $\delta/2$.
	
	For all $\epsilon > 0$ there exists $\delta$ with $2\leq \frac{1}{\delta} \in \NN$ and $N \in \NN$ such that
	\begin{equation}
	\PP(\exists I\in \mathbb{I}'_\delta \textrm{ such that } w_{X^n}(I)\geq \epsilon) \leq \epsilon \hspace{10mm} \forall n>N.
	\end{equation}
	
	There are only $\frac{2}{\delta}$ elements in $\mathbb{I}'_\delta$, so using a union bound it suffices to show that for sufficiently small $h$ and some $\delta$ we have
	\begin{equation}
	\sup_{I\in \mathbb{I}'_\delta} \PP(w_{X^n}(I)\geq 4\epsilon) \leq \frac{\delta\epsilon}{2}
	\end{equation}
	where a factor of 4 has been included purely for convenience later.
	
	We present the proof for $I=[0,\delta]$ but the same argument and bound will hold for all $I \in \mathbb{I}'_\delta$. We have that
	\begin{equation}	
	\PP(w_{X^n}(I)\geq 4\epsilon)\leq \PP\left( \sup_{t\leq \delta}X^n_t \geq 2\epsilon \right) + \PP\left( \inf_{t\leq \delta}X^n_t \leq -2\epsilon \right).
	\end{equation}
	We will bound the first term on the right with a bound that will also apply to the second term by symmetry.
	
	Unfortunately, Chebyshev's inequality is not strong enough to bound the first term sufficiently tightly. We will apply the Azuma-Hoeffding inequality which requires the following set-up. Let $X'^n_t=X^n_t-tB$ and note that this is a super-martingale. Fix $0<\alpha<\frac{1}{2}$, let $R_0=0$ and for $i\geq 1$ let $R_i$ be the first time $t>R_{i-1}$ such that $\lvert X'^n_t - X'^n_{R_{i-1}} \rvert \geq M_h^\alpha$.
	
	Firstly, we show that only about $\delta M_h^{-2\alpha}$ of the $R_i$ are less than $\delta$. Consider the distribution of $R_i-R_{i-1}$ conditional on $\FFF^n_{R_{i-1}}$, by the same argument used in the first condition we have the following for $l<\frac{M_h^\alpha}{4B}$
	\begin{align}
	\PP(R_i-R_{i-1}\leq l)&\leq \frac{2lA}{(M_h^\alpha - 2lB)^2} \\
	&\leq l\frac{8A}{M_h^{2\alpha}}.
	\end{align}
	From which we deduce that $R_i-R_{i-1}$ stochastically dominates the uniform distribution on $[0,\frac{M_h^{2\alpha}}{8A}]$ for sufficiently small $h$. An application of the Azuma-Hoeffding Inequality to uniform random variables gives
	\begin{equation}
	\PP(R_{\ceil{\frac{32A\delta}{M_h^{2\alpha}}}}\leq \delta)\leq \exp \left( -\frac{A\delta}{M_h^{2\alpha}} \right).
	\end{equation}
	Thus letting $J=\ceil{\frac{32A\delta}{M_h^{2\alpha}}}$ and $R$ be the minimum of $R_J$ and the first time $R_i$ such that $X'_{R_i}>\epsilon-\frac{M_h}{2}$,
	\begin{equation}
	\PP\left( \sup_{t\leq \delta}X'^n_t \geq 2\epsilon \right) \leq \exp \left( -\frac{A\delta}{M_h^{2\alpha}} \right) + \PP\left( \sup_{i \leq J }X'^n_{R_i} \geq 2\epsilon-M_h \right)
        \end{equation}
        and
	\begin{align}
	\PP&\left( \sup_{i \leq J }X'^n_{R_i} \geq 2\epsilon-M_h \right) \leq \nonumber \\
	& \hspace{10mm}\PP\left( X'^n_{R_J} \geq \epsilon-\frac{M_h}{2} \right) + \EE \left( \PP \left( X'^n_{R_J}-X'^n_{R} \leq -\epsilon+\frac{M_h}{2} \bigg| \FFF^n_{R} \right) \right).
	\end{align}
	We will bound the first term on the right of the last inequality, and note the second term can be bounded similarly. Let $X''^n_i=X'^n_{R_i}-iM_h$. Note that this is a discrete super-martingale with step size bounded by $M_h^\alpha+M_h$.
	\begin{align}
	\PP\left( X'^n_{R_J} \geq \epsilon-\frac{M_h}{2} \right)&\leq \PP\left( X''^n_J \geq \epsilon- (J+1)M_h \right) \\
	&\leq \exp \left( -\frac{(\epsilon-(J+1)M_h)^2}{2J(M_h^\alpha+M_h)^2} \right) \\
	&\leq \exp \left( -\frac{\epsilon^2}{65A\delta}\right) \textrm{ for sufficiently small } h
	\label{128bound}
	\end{align}
	where we have used the Azuma-Hoeffding inequality again. Bringing these bounds together gives that for a given $\delta$ we have for sufficiently small $h$ that
	\begin{equation}
	\label{finaltightnessbound}
	\sup_{I\in \mathbb{I}'_\delta} \PP(w_{X^n}(I)\geq 4\epsilon) \leq 2 \exp \left( -\frac{A\delta}{M_h^{2\alpha}} \right) + 4 \exp \left( -\frac{\epsilon^2}{65A\delta}\right).
	\end{equation}
	Thus, by choosing $\delta$ so that the second term is less than $\frac{\delta\epsilon}{4}$, and then choosing $N$ such that, for all $n\geq N$ we have, $h$ is sufficiently small that the bound \eqref{128bound} holds and the first term is less than $\frac{\delta\epsilon}{4}$, we can conclude that the second condition holds and the sequence $\mu_e^\eta$ is tight.
	
	By Prokhorov's theorem, we now know that every subsequence has a weakly convergent subsequence, and by standard arguments it suffices to show that the limit of every such sequence is $\mu_e^{a,b}$ (restricted to $[0,1]$). Let $\mu$ be the limit of such a subsequence and $X$ be distributed according to $\mu$.
	
	We now show that $X$ is a solution of the SDE \eqref{SDE}. Let $(\FFF_t)_{t\geq s}$ be the completion of the filtration generated by $X$, and let $W$ be given by
	\begin{equation}
	W_t=\int_0^t \frac{1}{\sigma(s,X_s)} dX_s-\int_0^t \frac{b(s,X_s)}{\sigma(s,X_s)} ds.
	\end{equation}
	
	Note continuity of $X$ follows from the bound \eqref{continuitycondition}, and so $\FFF$ is right-continuous and thus satisfies the usual conditions. It is immediate that $X_s=x$ and Equation \eqref{SDEintegralform} holds by the definition of $W$.
	
	The identities \eqref{firstexpectationbound} and \eqref{secondexpectationbound} show in the limit $n\to \infty$ that both $X$ and the quadratic variation of $X$ are a.s. bounded on each compact interval. The same argument used to get these identities can also be used to find that
	\begin{equation}
	\EE(W_t-W_s| \FFF_s)=0
	\end{equation}
	and
	\begin{equation}
	\EE((W_t-W_s)^2| \FFF_s)=t-s.
	\end{equation}
	From the definition of $W$ and the continuity of $X$, we can deduce $W$ is continuous a.s., putting this together with the above expectations we can conclude by L\'evy-Characterization that $W$ is a $(\FFF_t)$-Brownian motion.
	
	Thus $X$ solves \eqref{SDE} and has the required law.
\end{proof}

Define $\lambda_h$ to be the infimum of $\lambda$ such that,

\begin{equation}
\label{localisation}
\lambda \leq \abs{x-y} \leq 1-\lambda \implies
\frac{1}{h}\EE (\abs{\tilde{F}_{h,t}(x)\tilde{F}_{h,t}(y)})<\lambda \hspace{10mm} \forall t.
\end{equation}

\begin{prop}
	\label{countableconv}
	Under the conditions of Proposition \ref{singleconv} and that $\lambda_h \rightarrow 0$, we have $\mu^\eta_E \rightarrow \mu_E^{a,b}$ weakly on $D_E$.
\end{prop}

\begin{proof}
	We write $X^k$ for $X^{e_k}$. The family of laws on $D_E$ is tight as each family of marginal laws on $D_{e_k}$ is tight. Let $\mu$ be a weak limit law for $\mu^\eta_E$, then for all $j,k$ and all $t>s\geq s_j\vee s_k$, letting $\EE^*(\cdot )=\EE(\cdot \mid t_{N_s+1}\leq t<t_{N_s+2},\FFF_s )$ we have,
	\begin{align}
	&\EE( X^j_tX^k_t-X^j_sX^k_s\mid \FFF_s) \\
	=&\frac{t-s}{h} \EE^*(F_{N_s+1}(X^j_s)F_{N_s+1}(X^k_s)-X^j_sX^k_s)+O\left(\frac{(t-s)^2}{h^2}\right) \\
	\begin{split}
	=&\frac{t-s}{h} \EE^*(\tilde{F}_{N_s+1}(X^j_s)X^k_s+\tilde{F}_{N_s+1}(X^k_s)X^j_s +\tilde{F}_{N_s+1}(X^k_s)\tilde{F}_{N_s+1}(X^j_s)) \\&+O\left(\frac{(t-s)^2}{h^2}\right)
	\end{split}\\
	\begin{split}
	=& \int_{s}^{t}b_h(r,X^j_s)dr X^k_s + \int_{s}^{t}b_h(r,X^k_s)dr X^j_s  \\&+\frac{(t-s)}{h}\EE^*(\tilde{F}_{N_s+1}(X^k_s)\tilde{F}_{N_s+1}(X^j_s)) +O\left(\frac{(t-s)^2}{h^2}\right)		\end{split}\\
	\begin{split}
	=& \EE\left(\int_{s}^{t} b(r,X^j_r)X^k_r+b(r,X^k_r)X^j_r dr\mid \FFF_s\right) + E_1 \\&+\frac{(t-s)}{h}\EE^*(\tilde{F}_{N_s+1}(X^k_s)\tilde{F}_{N_s+1}(X^j_s)) +O\left(\frac{(t-s)^2}{h^2}\right).
	\end{split}
	\end{align}
	Where we have (by the same method used to bound $E_3$ in Proposition \ref{singleconv}) that
	\begin{align}
	\abs{E_1} &= O\left((t-s)^2\right)
	\end{align}
	and provided $\abs{X^j_s-X^k_s}\geq \lambda_h$ (distance considered modulo one) we have
	\begin{equation}
	\left| \frac{(t-s)}{h}\EE^*(\tilde{F}_{N_s+1}(X^k_s)\tilde{F}_{N_s+1}(X^j_s)) \right| \leq (t-s)\lambda_h.
	\end{equation}
	So for $(t-s)^{\frac{1}{2}}\ll h \ll 1$ we have
	\begin{align}
	&\EE\left(X_t^jX_t^k-X_s^jX_s^k \big| \FFF_s, |X_s^j-X_s^k|	\geq \lambda_h \right)\\
	=&\EE\left(\int_{s}^{t} b(r,X^j_r)X^k_r+b(r,X^k_r)X^j_r dr\right)+o(t-s).
	\end{align}
	
	Hence, breaking $[s_j\vee s_k,\infty)$ into intervals of length $t-s$ and taking the limit as $t-s$ and $h$ go to $0$, gives that the process
	\begin{equation}
	X^j_tX^k_t-\int_{s_j \vee s_k}^t X^k_sb(s,X^j_s)+X^j_sb(s,X^k_s)ds,
        \end{equation}
        stopped at time $T^{jk}$ is a martingale. Further, this process must be continuous because Proposition \ref{singleconv} tells us that $X_t^j$ and $X_t^k$ are continuous.
	We know from Proposition \ref{singleconv} that, under $\mu$, both \mbox{$(X^k_t-\int_{s_k}^t b(X^k_s) ds)_{t\geq s_k}$} and
	\begin{equation}
	\left((X^k_t)^2-2\int_{s_k}^t X^k_sb(s,X^k_s)ds-\int_{s_k}^t a(s,X^k_s) ds \right)_{t\geq s_k}
	\end{equation}
	are continuous local martingales.
	
	It remains to show that $X_t^j-X_t^k$ is constant for $t\geq T^{jk}$ after which the result follows from Proposition \ref{countableweb}. Let \mbox{$Y_t=X^j_t-X^k_t$} and assume w.l.o.g that $Y_0>0$ and $Y_{T^{jk}}=0$. The process $Y$ inherits the property of not changing sign as our disturbances are order preserving. Given $R\in \RR$ and $\epsilon >0$ localize $Y$ using the stopping time \mbox{$S=\inf\{ t: Y_t>1 \textrm{ or } t>R \}$} and note that
	\begin{equation}
	\EE \abs{Y^S_{T^{jk}+t}} \leq \int_{T^{jk}}^{T^{jk}+t} \EE L\abs{Y^S_s} ds = L\int_{T^{jk}}^{T^{jk}+t} \EE \abs{Y^S_s} ds.
	\end{equation}
	Recall $L$ is the Lipschitz constant of $b$. So, by Gronwall's inequality, $\EE \abs{Y^S_{T^{jk}+t}}$ is identically $0$, up to time $t=R$. So $Y_t=0$ for all $t>T^{jk}$ a.s. and we are done.
\end{proof}

Let $E=(e_k:k\in \NN)$ be an enumeration of $\QQ^2$. Write $Z^{E,\pm}$ for the maps $D^\circ(\RR,\DDD) \rightarrow D_E$ given by $Z^{E,\pm}=(Z^{e_k,\pm}:k\in\NN)$. Write $Z^E=Z^{E,+}$. The following result is a criterion for weak convergence on $D^\circ(\RR,\DDD)$, and is Theorem 5.1 of \cite{NT}.

\begin{thm}
	\label{countableconvimpliesconv}
	Let $(\mu_n:n\in \NN)$ be a sequence of Borel probability measures on $D^\circ (\RR,\DDD)$, and let $\mu$ be a Borel probability measure on $C^\circ (\RR,\DDD)$. Assume that $Z^{E,-}=Z^{E,+}$ holds $\mu_n$-almost surely for all $n$ and $\mu$-almost surely. Assume further that $\mu_n \circ (Z^E)^{-1}\rightarrow \mu \circ (Z^E)^{-1}$ weakly on $D_E$. Then $\mu_n\rightarrow \mu$ weakly on $D^\circ (\RR,\DDD)$.
\end{thm}

The following result generalizes Theorem 6.1 of \cite{NT} to the case of varying drift and diffusivity. It is immediate from Proposition \ref{countableconv} and Theorem \ref{countableconvimpliesconv}.
\begin{thm}
	\label{flowconv}
	Given a family of distributions $F$ along with $a,b$, Lipschitz in space measurable in time, obeying Equations \eqref{convbh}-\eqref{jumpbound} and with $\lambda_h \rightarrow 0$ then the convergence
	\begin{equation}
	\mu^F_A \rightarrow \mu^{a,b}_A \textrm{ weakly on } D^\circ(\RR,\DDD) \textrm{ as } h \rightarrow 0
        \end{equation}
        holds.
\end{thm}

\vspace{10pt}
\section{Proof of Theorem \ref{reversalidentify} using Disturbance Flows}
\label{DisturbanceProof}
\vspace{10pt}

In this section, we identify the time-reversal of a generic disturbance flow. We then apply this identification to an explicit sequence of flows and, as the limit of the reversals must be the reversal of the limit, we can deduce Theorem \ref{reversalidentify}.

The following proposition is a generalization of the first half of Proposition 7.1 of \cite{NT}, which can be recovered by assuming that $b_h\equiv 0$ and $a_h\equiv 1$.

\begin{prop}
	Set $G_{h,t}=F^{-1}_{h,-t}$.
	The time-reversal of a disturbance flow with disturbance $F_h$ is a disturbance flow with disturbance $G_h$, for all $h$. Thus $\hat\mu^{F_h}_A = \mu^{G_h}_A$, for all $h$.
\end{prop}
\begin{proof}
	The proof is very close to the second half of the proof of proposition 7.1 of \cite{NT}.
	
	Set $m$ and $n$ to be the minimal and maximal values taken by $N_t$ at jumps in $I$. Also, take $-\hat{n}$ and $-\hat{m}$ to be the minimal and maximal values taken by $N_t$ at jumps in $-I$. Then, we can define a disturbance flow $\Phi$ with disturbance $F_h$, by
	\begin{equation}
	\Phi^{\pm}_I=F^{\pm}_{h,t_n}\circ \dots \circ F^{\pm}_{h,t_{m}}.
	\end{equation}
	Then
	\begin{equation}
	\hat\Phi^{\pm}_I = G^{\pm}_{h,-t_{-\hat{n}}} \circ \dots \circ G^{\pm}_{h,-t_{-\hat{m}}}.
	\end{equation}
	By the properties of the Poisson process $(-t_{-\hat{m}},\dots ,-t_{-\hat{n}})$ is equal in distribution to $(t_{m},\dots ,t_{n})$, so $\hat\Phi$ is a disturbance flow with disturbance $G_h$.
\end{proof}

In \cite{NT}, it is then shown for $a\equiv 1$ and $b \equiv 0$ that $\mu_A$ is invariant under time-reversal. We generalize this result to Theorem \ref{reversalidentify}.

\begin{thm:reversal}
	If $a$ has spatial derivative $a'$ and $a$, $b$ and $a'$ are uniformly bounded on compacts in time and $L$-Lipschitz in space then
	\begin{equation}
	\hat\mu_A=\nu_A:=\mu^{a^\nu,b^\nu}_A
	\end{equation}
	where 	$ a^\nu(t,x)=a(-t,x)$, $b^\nu(t,x)=-b(-t,x)+a'(-t,x)/2$ and $\hat{\mu}_A$ is the time reversal of $\mu_A$.
\end{thm:reversal}
\begin{proof}
	The proof is based on the fact that given a family $(F_h)_{h>0}$ (satisfying the conditions of Proposition \ref{singleconv}) we have that: $ \mu^{F^{-1}_{h,-t}}_A=\hat\mu^{F_h}_A \rightarrow \hat\mu_A $. It thus suffices to show for some specific family $(F_h)_{h>0}$ that $ \mu^{F^{-1}_{h,-t}}_A \rightarrow \nu_A$. This is true by Theorem \ref{flowconv} if $(F^{-1}_{h,-t})_{h>0}$ satisfies the conditions that we put on $F$, but with $a(t,x)$ and $b(t,x)$ replaced by $a^\nu(-t,x)$ and $b^\nu(-t,x)$. Let $\hat{a}_h$, $\hat{a}$, $\hat{b}_h$ and $\hat{b}$ be defined from $F^{-1}$ as $a_h$ and $b_h$ are defined from $F$.

        For every fixed $h$, consider the sequence $t_n$, and let $\theta_{h,t_n}$ be i.i.d. uniform random variables on $[0,1]$. We will write $\theta$ as shorthand for $\theta_{h,t_n}$, and for the remainder of this proof $t$ will refer to an element of $\{t_n:n\in\NN\}$. Let
	\begin{equation}
	r_{\theta,t}=\frac{h^\frac{2}{3}}{2}\left(b\left(t,\theta-\frac{1}{2}\right)-a'\left(t,\theta-\frac{1}{2}\right)\right)
	\end{equation}
	and
	\begin{equation}
	w=\left( \frac{3a(t,\theta)h}{2} \right)^{\frac{1}{3}}.
	\end{equation}
	Then, for sufficiently small $h$, we consider the family of disturbances given by setting,
	\begin{equation}
	F_{h,t}(x)=\begin{cases}
	x+r_{\theta,t} & (x-\theta)\in (\frac{1}{2}-h^{\frac{1}{3}},\frac{1}{2}+h^{\frac{1}{3}}) \\
	\frac{1}{2}+h^\frac{1}{3}+r_{\theta,t}+\theta& (x-\theta)\in (\frac{1}{2}+h^\frac{1}{3},\frac{1}{2}+h^\frac{1}{3}+r_{\theta,t}) \\
	\frac{1}{2}-h^\frac{1}{3}+r_{\theta,t}+\theta& (x-\theta)\in (\frac{1}{2}-h^\frac{1}{3}+r_{\theta,t},\frac{1}{2}-h^\frac{1}{3}) \\
	\theta & (x-\theta) \in (-w,w) \\
	x & \textrm{otherwise}.
	\end{cases}
	\end{equation}
	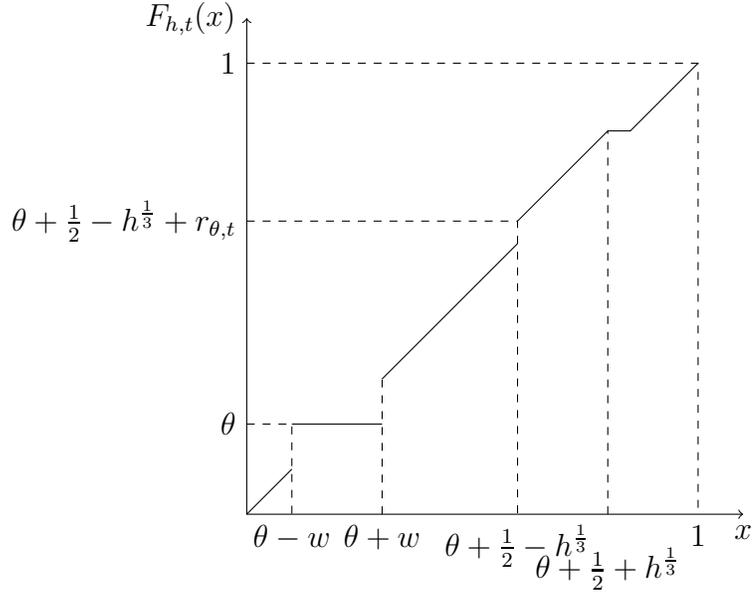
\begin{figure}
		\begin{tikzpicture}[scale=0.6]
		\draw [<->] (11,0) -- (0,0) -- (0,11);
		\node [left] at (0,11) {$F_{h,t}(x)$};
		\node [below] at (11,0) {$x$};
		\draw (0,0) -- (1,1) (1,2) -- (3,2) (3,3) -- (6,6) (6,6.5) -- (8,8.5) -- (8.5,8.5) -- (10,10);
		\draw [dashed] (0,10) -- (10,10) --(10,0);
		\node [left] at (0,10) {1};
		\node [below] at (10,0) {1};
		\draw [dashed] (0,2) -- (1,2) -- (1,0) (3,0) -- (3,3);
		\node [left] at (0,2) {$\theta$};
		\node [below] at (1,0) {$\theta-w$};
		\node [below] at (3,0) {$\theta+w$};
		\draw [dashed] (0,6.5) -- (6,6.5) -- (6,0) (8,0) -- (8,8.5);
		\node [left] at (0,6.5) {$\theta+\frac{1}{2}-h^\frac{1}{3}+r_{\theta,t}$};
		\node [below] at (6,0) {$\theta+\frac{1}{2}-h^\frac{1}{3}$};
		\node [below] at (8,-0.5) {$\theta+\frac{1}{2}+h^\frac{1}{3}$};
		\end{tikzpicture}
		\caption{\label{specificDisturbance}An example from the specific family of disturbances used in this proof.}
	\end{figure}
	Note that at least one of the intervals specified is empty, depending on the sign of $r_{\theta,t}$. An example from this family is graphed in Figure \ref{specificDisturbance}. 
	
	Note that $\lambda_h\rightarrow 0$ both for $F$, as originally defined, and with $F^{-1}$ substituted for $F$. Note that, the disturbance of size $r_{\theta,t}$ is negligible in computing $\lambda_h$ as it is $O(h^{\frac{2}{3}})$ in magnitude $O(h^{\frac{1}{3}})$ in width and always multiplied by something of size $O(h^{\frac{1}{3}})$ in the definition of $\lambda_h$.
	The first three cases in the above definition also contribute nothing to either $\lim_{h\rightarrow 0}a_h$ or $\lim_{h\rightarrow 0}\hat{a}_h$, and their contribution to $\lim_{h\rightarrow 0}b_h$ is exactly the negative of their contribution to $\lim_{h\rightarrow 0}\hat{b}_h$. So it suffices to prove that the proposition holds for the case $b=a'$, i.e. the case where $r_{\theta,t}\equiv 0$.

	We write $w_{\pm}$ for the largest offsets from $x$ a disturbance can have whilst not mapping $x$ to itself. For sufficiently small $h$ they are given by the implicit equation
	\begin{equation}
	w_\pm=\left( \frac{3a(t,x\pm w_\pm)h}{2} \right)^{\frac{1}{3}}.
	\end{equation}
	We expanding this by Taylor's theorem and by substituting the right hand side in for $w_\pm$ on the right hand side. Letting
	\begin{equation}
	c=c_h(x)=\left( \frac{3a(t,x)h}{2} \right)^{\frac{1}{3}}
	\end{equation}
	this expansion gives
	\begin{equation}
	w_\pm = c \pm \frac{a'(t,x)c^2}{3a(t,x)} + o(h^\frac{2}{3}).
	\end{equation}
	For the rest of this proof unless otherwise specified $a$, $a'$ and $b$ are assumed to be evaluated at $(t,x)$. Note that the $o(h^\frac{2}{3})$ term is small uniformly in $t$, this will allow use to conclude uniform convergence. All following uses of little $o$ notation in this proof have implied constants independent of $t$, in all cases this follows from the fact that $b$ and $a$ are bounded uniformly.

        We can now calculate
	\begin{align}
	a_h&=\frac{1}{h}\EE\left(\tilde{F}_{h,t}(x)^2\right) \\
	&=\frac{1}{h}\int_{-w_-}^{w_+}\alpha^2 d\alpha \\
	&=\frac{1}{3h}\left( w_+^3+w_-^3 \right) \\
        &=\frac{2c^3}{3h} +o(1) \\
        &=a+o(1) \\
	&\rightarrow a
	\end{align}
	and
	\begin{align}
	b_h&=\frac{1}{h}\EE\left(\tilde{F}_{h,t}(x)\right) \\
	&=\frac{1}{h}\int_{-w_-}^{w_+}\alpha d\alpha \\
        &=\frac{1}{2h}\left( w_+^2-w_-^2 \right) \\
        &=\frac{1}{2h}\left(\left(c^2+\frac{2a'c^3}{3a}+o(h)\right)-\left(c^2-\frac{2a'c^3}{3a}+o(h)\right)\right)\\
        &=\frac{2a'c^3}{3ha} + o(1) \\
        &=a'+o(1) \\  
	&\rightarrow a'.
	\end{align}
	By Taylor and binomial expansion we also get
	\begin{equation}
	c_h(x+\alpha)=c_h(x)\left( 1+\frac{a'\alpha}{3a} \right)+ o(h^\frac{1}{3}\alpha).
	\end{equation}
	Which allows us to calculate,
	\begin{align}
          \hat{a}_h(-t,x)&=\frac{1}{h}\int_{-w_-}^{0}(\alpha+c_h(x+\alpha))^2 d\alpha+\frac{1}{h}\int_{0}^{w_+}(\alpha-c_h(x+\alpha))^2 d\alpha \\
	&=\frac{1}{h}\int_{-c}^0 \left(\alpha^2+2\alpha c+c^2 \right)d\alpha +\frac{1}{h}\int_0^c \left(\alpha^2-2\alpha c+c^2\right) d\alpha + o(1) \\
	&=\frac{2}{h}\int_0^c \left(\alpha^2-2\alpha c +c^2 \right)d\alpha + o(1)\\
        &=\frac{2}{h}\left(\frac{c^3}{3}-c^3+c^3\right) + o(1) \\
        &=a+o(1) \\
	&\rightarrow a
        \end{align}
	and
	\begin{align}
	\hat{b}_h(-t,x)&=\frac{1}{h}\int_{-w_-}^{0}\left(\alpha+c_h(x+\alpha)\right) d\alpha+\frac{1}{h}\int_{0}^{w_+}\left(\alpha-c_h(x+\alpha)\right) d\alpha \\
	&=b_h+\frac{c}{h}\int_{-w_-}^{0} \left(1+\frac{a'\alpha}{3a}+o(\alpha)\right) d\alpha -\frac{c}{h}\int_{0}^{w_+} \left(1+\frac{a'\alpha}{3a}+o(\alpha)\right) d\alpha \\
	&=b_h+\frac{c}{h} \left( w_--w_+ -\frac{a'}{6a}(w_-^2+w_+^2) +o(h^\frac{2}{3}) \right) \\
	&=b_h+\frac{c}{h}\left( -\frac{2a'c^2}{3a}-\frac{a'c^2}{3a} \right) +o(1) \\
	&=b_h-a'-\frac{a'}{2} +o(1)\\
        &=-\frac{a'}{2}+o(1) \\
        &\rightarrow -\frac{a'}{2}.
	\end{align}
	So the result holds.
\end{proof}

The following corollary is similar to Corollary 7.3 of \cite{NT} (and with an almost identical proof) in that it gives weak convergence for paths running both forward and backward from a given sequence of points. First we define the notation for this result.

Given $e=(s,x)\in \RR^2$, define $\bar{D}_e=\{ \xi \in D(\RR,\RR) : \xi_s=x \}$ and for $E=(e_k:k\in \NN)$ set $\bar{D}_E=\prod_{k=1}^{\infty}\bar{D}_{e_k}$. For $\phi \in D^\circ(\RR,\DDD)$, define
\begin{equation}
\bar{Z}_t^{e,\pm}(\phi)=\begin{cases}
\phi_{(s,t]}^\pm (x), &t\geq s, \\
(\phi^{-1})_{(t,s]}^\pm (x), &t<s.
\end{cases}
\end{equation}
Then $\bar{Z}^{e,\pm}(\phi) \in \bar{D}_e$ and extends $Z^{e,\pm}(\phi)$, from $[s,\infty)$ to the whole of $\RR$. Let $\eta_h$ denote the law of $F_h$. For all $e\in \RR^2$, we have $\bar{Z}^{e,-}=\bar{Z}^{e,+}$ almost everywhere on $D^\circ (\RR,\DDD)$ wth respect to both $\mu^{a,b}_A$ and $\mu^{\eta_h}_A$, for all $h$. So we drop the $\pm$. Denote by $\bar{\mu}^{\eta_h}_E$ the law of $(\bar{Z}^{e_k}:k\in \NN)$ on $\bar{D}_E$ under $\mu^{\eta_h}_A$ and by $\bar{\mu}^{a,b}_E$ the corresponding law under $\mu^{a,b}_A$.
\begin{cor}
	\label{bidirectionalflow}
	$\bar{\mu}^{\eta_h}_E \rightarrow \bar{\mu}^{a,b}_E$ weakly on $\bar{D}_E$,
\end{cor}
\begin{proof}
	Given $\phi$ with law $\mu^{a,b}_A$, we have that almost surely
	\begin{equation}
	\bar{Z}^{(s,x\pm \delta),+}(\phi) \rightarrow \bar{Z}^{(s,x),+}(\phi)
	\end{equation}
	uniformly on $\RR$ as $\delta \rightarrow 0$. We also have $\phi\in C^\circ(\RR,\DDD)$ almost surely and it follows that $\bar{Z}^{(s,x),+}$ is continuous at $\phi$ almost surely. Thus, the result holds as we already know the convergence holds component wise.
\end{proof}

\vspace{10pt}
\section{Proof of Theorem \ref{reversalidentify} without Disturbance Flows}
\label{DirectProof}
\vspace{10pt}

In this section we first prove a version of Theorem \ref{reversalidentify} with the extra hypothesis that $a$ and $b$ are Lipschitz in time. Then we use an approximation argument to show Theorem \ref{reversalidentify} in the general case.

\begin{thm}
	\label{reversalLipschitz}
	If $a$ has spatial derivative $a'$ and $a$, $b$ and $a'$ are Lipschitz in both time and space then
	\begin{equation}
	\hat\mu_A=\nu_A:=\mu^{a^\nu,b^\nu}_A
	\end{equation}
	where 	$ a^\nu(t,x)=a(-t,x)$, $b^\nu(t,x)=-b(-t,x)+a'(-t,x)/2$ and $\hat{\mu}_A$ is
	the time reversal of $\mu_A$.
\end{thm}
\begin{proof}
	Let $\phi \sim \mu_A$. It suffices to show that the restriction of $\hat\phi$ to $E$ given by $Z^{E,+}(\hat\phi)$, which we shall call $\hat\phi_E$, has distribution $\nu_E$, for each countable set $E\subset \RR\times\RR$. The distribution $\nu_E$ is characterised by its restriction to two point motions by Theorem \ref{webexists}.
	
	Coalescence of two motions follows immediately from the definition of time-reversal. As does the continuity of a single motion.
	
	As $\phi_{ts}$ and $\phi_{su}$ are independent for $s\in (u,t)$, we have the Markov property. Thus, by Donsker's Invariance Principle, we can identify the two point motion from just the mean and covariance matrix of small increments.
	
	First, we consider each one point motion separately. We will proceed by relating the backward and forward flows, then, noting that increments of the forward process are small, we approximate $a$ and $b$ on an interval that the forward process almost surely won't leave in such a way as to make exact calculations possible. Then we check that the incurred error is small using that $a$ and $b$ are Lipschitz in time, and that the exact calculations give the required answer. Finally, we will show that the increments of each process are independent, conditional on an event of large probability, and so the covariances are small.
	
	We have the relation,
	\begin{equation}
	\PP\left(\hat{\phi}_{t+h,t}(y)<x\right)=\PP(\phi_{-t,-t-h}(x)>y)
	\end{equation}
	which we can use to determine the distribution of $\hat{\phi}_{t+h,t}(y)$ if we first understand the distributions of the variables $\phi_{-t,-t-h}(x)$.
	
	To study these variables, we first show that the forward paths are localised. Start by noting
	\begin{align}
          \PP\left(\sup_{0<\delta t<h}|\phi_{t+\delta t,t}(x)-x|>h^{\frac{1}{2}-\epsilon}\right) \leq& \PP\left(\sup_{0<\delta t<h}\phi_{t+\delta t,t}(x)-x>h^{\frac{1}{2}-\epsilon}\right) \\
          &+\PP\left(\inf_{0<\delta t<h}\phi_{t+\delta t,t}(x)-x<-h^{\frac{1}{2}-\epsilon}\right).
        \end{align}
        Each of these terms can be bounded in the same way. To bound the first term, consider the process $\phi_{t+\delta t,t}(x)-b^*\delta t$ parametrised by $\delta t$. This is a supermartingale with diffusivity bounded by $a^*$, and thus by the reflection principle
        \begin{align}
          \PP\left(\sup_{0<\delta t<h}\phi_{t+\delta t,t}(x)-x>h^{\frac{1}{2}-\epsilon}\right)&\leq 2\Phi\left(-\frac{h^{\frac{1}{2}-\epsilon}-b^*h}{2(a^*h)^{\frac{1}{2}}}\right).
        \end{align}
        Thus we can derive that,
        \begin{align}
          &\PP\left(\sup_{0<\delta t<h}|\phi_{t+\delta t,t}(x)-x|>h^{\frac{1}{2}-\epsilon}\right) \\
          \leq& 4\Phi\left(-\frac{h^{\frac{1}{2}-\epsilon}-b^*h}{2(a^*h)^\frac{1}{2}}\right) \\
	\leq& \exp\left(-C(a^*,b^*)h^{-2\epsilon}\right) \hspace{10mm} \textrm{for sufficiently small } h
	\end{align}
	where $C$ is positive and independent of $h$ and $t$.
	
	Now we approximate $a$ and $b$ by $\tilde{a}$ and $\tilde{b}$ which, on the interval $[y-2h^{\frac{1}{2}-\epsilon},y+2h^{\frac{1}{2}-\epsilon}]$, are given by,
	\begin{equation}
	\tilde{a}(s,x)=\frac{1}{4a}\left((x-y)a'+2a\right)^2
	\end{equation}
	and
	\begin{equation}
	\tilde{b}(s,x)=\frac{b}{2a}\left((x-y)a'+2a\right).
	\end{equation}
	Where we have written $a$ for $a(t,y)$, $a'$ for $a'(t,y)$ and $b$ for $b'(t,y)$. We then extend $\tilde{a}$ and $\tilde{b}$ to functions on the circle which are both $\tilde{L}$-Lipschitz continuous and $\tilde{L}$-Lipschitz differentiable, for some $\tilde{L}$. For sufficiently small values of $h$, this extension can and will be chosen so that $a_*/2\leq \tilde{a}\leq 2a^*$. Note that for all $s$, $a(t,y)=\tilde{a}(s,y)$, $a'(t,y)=\tilde{a}'(s,y)$ and $b(t,y)=\tilde{b}(s,y)$, this will turn out to make them sufficiently good approximations.
	
	We now approximate the diffusion process $\phi_{t+\delta t,t}(x)$ for each $x\in [y-h^{\frac{1}{2}-\epsilon},y+h^{\frac{1}{2}-\epsilon}]$ by a diffusion process $X_{\delta t}$ started from $x$, with drift $\tilde{b}$ and diffusivity $\tilde{a}$, but driven by the same Brownian motion $B_{\delta t}$ as $\phi_{t+\delta t,t}(x)$. Note that $\tilde{a}$ and $\tilde{b}$ are constant with respect to time. Let $G$ be the event,
	\begin{equation}
	\left\{\sup_{0<\delta t<h}|\phi_{t+\delta t,t}(x)-x|<h^{\frac{1}{2}-\epsilon}\right\}\cap \left\{\sup_{0<\delta t<h}|X_{\delta t}-x|<h^{\frac{1}{2}-\epsilon}\right\}
	\end{equation}
	and note the second event in this union has probability bounded like the first, so $\PP(G)=1-O\left( e^{-Ch^{-2\epsilon}}\right)$. Note also that on this event, $X$ and $\phi_{t+\delta t,t}(x)$ stay within the interval we explicitly defined $\tilde{a}$ and $\tilde{b}$ on.

	On this event the error in the approximation is given by,
	\begin{align}
	\Delta_{\delta t}&:=X_{\delta t}-\phi_{t+\delta t,t}(x)\\
	&=\int_0^{\delta t} \left(\tilde{b}(t,X_{u})-b(t+u,\phi_{t+u,t}(x))\right) du \\ &\hspace{5mm}+\int_0^{\delta t} \left(\sqrt{\tilde{a}(t,X_{u})}-\sqrt{a(t+u,\phi_{t+u,t}(x))}\right) dB_u.
	\end{align}
	We have that if $E_{h}:=\sup_{\delta t<h} |\Delta_{\delta t}|$ then,
	\begin{align}
	E_{h}&\leq\sup_{\delta t<h}\int_0^{\delta t} &&\big( |\tilde{b}(t,X_{u})-\tilde{b}(t,\phi_{t+u,t}(x))|
	+|\tilde{b}(t,\phi_{t+u,t}(x))-b(t,y)|\\ &&&+|b(t,y)-b(t,\phi_{t+u,t}(x))| \\ &&& +|b(t,\phi_{t+u,t}(x))-b(t+u,\phi_{t+u,t}(x))|\big) du \\
	&\hspace{5mm}+\biggl|\int_0^{\delta t} && \bigg(\sqrt{\tilde{a}(t,X_{u})}-\sqrt{\tilde{a}(t,\phi_{t+u,t}(x))} +\sqrt{\tilde{a}(t,\phi_{t+u,t}(x))}-\sqrt{a(t,y)} \\ &&& +\sqrt{a(t,y)}-\sqrt{a(t,\phi_{t+u,t}(x))} \\ &&&+\sqrt{a(t,\phi_{t+u,t}(x))}-\sqrt{a(t+u,\phi_{t+u,t}(x))}\bigg) dB_u\biggr|.
	\end{align}
        The first integrand is bounded, on $G$, by $2\tilde{L}h^{\frac{1}{2}-\epsilon}+2Lh^{\frac{1}{2}-\epsilon}+2\tilde{L}h^{\frac{1}{2}-\epsilon}+Lh \leq 4(L+\tilde{L})h^{\frac{1}{2}-\epsilon}$. The first integral is therefore bounded by $4\delta t(L+\tilde{L})h^{\frac{1}{2}-\epsilon}$. To bound the second integrand we first observe that the square root function is Lipschitz with some constant $L_s$ on the interval $[a_*/2,2a^*]$. Secondly, note that we can achieve a stronger bound than for the first integrand as $a'$ is Lipschitz in $x$. The second integrand is thus bounded, on $G$, by $d_1:=2\tilde{L}L_sh^{\frac{1}{2}-\epsilon} +2LL_sh^{1-2\epsilon}+2\tilde{L}L_sh^{1-2\epsilon} +LL_sh$. The second integral is therefore a continuous martingale with diffusivity bounded by $d_1$. It can therefore be written as a time change of a standard Brownian motion, such that $h$ can correspond to a time no later than $hd_1$. Equivalently, there exists a Brownian motion $B'$ such that, on the event $G$, $E_h\leq \Delta'_{h}$ where
	\begin{align}
	\Delta'_{h}&=4h(L+\tilde{L})h^{\frac{1}{2}-\epsilon}+\sup_{\delta t<h}\left| \int_0^{\delta t} d_1 dB'_u\right|.
	\end{align}
        Consider the event $G'=\{\Delta'_{h}<h^{1-2\epsilon}\}$. The probability of this event is $1-O(e^{-Ch^{-2\epsilon}})$. This isn't quite a strong enough bound due to the $h^{\frac{1}{2}-\epsilon}$ term in $d_1$. However, as that term is proportional to the bound we have on $X_u-\phi_{t+u,t}(x)$ and this result provides a stronger bound, we can bootstrap this argument. To that effect note that, on $G\cap G'$, the second integrand is bounded by $2\tilde{L}L_sh^{1-2\epsilon}+2(L+\tilde{L})L_sh^{1-2\epsilon}+LL_sh \leq 4(L+\tilde{L})L_sh^{1-2\epsilon}$. Thus, as before, there exists a Brownian motion $B''$ such that, on the event $G\cap G'$, $E_{h}\leq \Delta''_h$ where
	\begin{align}
	\Delta''_{h}&=4h(L+\tilde{L})h^{\frac{1}{2}-\epsilon}+\sup_{\delta t<h}\left| \int_0^{\delta t} 4(L+\tilde{L})L_sh^{1-2\epsilon} dB''_u\right|.
	\end{align}
	Finally consider $G''=\{\Delta''_h<h^{\frac{3}{2}-3\epsilon}\}$ and note that the probability of this event, conditioned on $G\cap G'$ is $1-O(e^{-Ch^{-2\epsilon}})$.
	Thus we can conclude that
	\begin{align}
	\label{Deltabound}
          \PP(|\Delta_{h}|>h^{\frac{3}{2}-3\epsilon})&\leq \PP(E_h> \Delta''_h \textrm{ or } \Delta''_h\geq h^{\frac{3}{2}-3\epsilon}) \\
                                                     &\leq 1-\PP(G\cap G'\cap G'') \\
                                                     &=O\left(e^{-Ch^{-2\epsilon}}\right).
	\end{align}
	This result suffices to control the error of the approximation.
	
	Next, we calculate the distribution of $X_h$. Note that on the event $G$, we have, for some Brownian motion $W$, that
	\begin{equation}
	dX_t=\frac{ba'}{2a}\left(X_t-y+\frac{2a}{a'}\right)dt+\frac{a'}{2\sqrt{a}}\left(X_t-y+\frac{2a}{a'}\right)dW_t.
	\end{equation}
	Where we are again writing $a$ for $a(t,y)$, $a'$ for $a'(t,y)$ and $b$ for $b'(t,y)$.

        We consider three separate possibilties here; $a'>0$, $a'=0$, or $a'<0$. The $a'<0$ case we omit as it follows from the $a'>0$ case by symmetry. If $a'=0$ then define $f(x)=x/\sqrt{a}$ otherwise define $f(x)=\frac{2\sqrt{a}}{a'}\log(x-y+\frac{2a}{a'})$. Note that as $a'\leq L$ and $a>a_*$ the logarithm is well defined, so long as $h$ is sufficiently small. Either way an application of It\={o}'s lemma gives that
	\begin{equation}
	df(X_t)=\frac{1}{\sqrt{a}}\left(b-\frac{a'}{4}\right)dt+dW_t.
	\end{equation}
	The choices for $\tilde{a}$, $\tilde{b}$ and $f$ were made so that this equation has constant coefficients. Thus $f(X_h)$ is normally distributed with mean $f(x)+\frac{h}{\sqrt{a}}\left(b-\frac{a'}{4}\right)$ and variance $h$. So we can calculate that $\PP(X_h>y)=\PP(f(X_h)>f(y))=F_y(x)+O(e^{-Ch^{-2\epsilon}})$, where
	\begin{align}
	F_y(x)&:=\Phi\left(\frac{f(x)-f(y)}{\sqrt{h}}+\sqrt{\frac{h}{a}}\left(b-\frac{a'}{4}\right)\right) \\
	&=\Phi\left(\frac{2\sqrt{a}}{a'\sqrt{h}}\log\left(1+\frac{a'}{2a}(x-y)\right)+\sqrt{\frac{h}{a}}\left( b-\frac{a'}{4}\right)\right).
	\end{align}
	This implies that
	\begin{equation}
	\PP(|X_h-x|>h^{\frac{1}{2}-\epsilon})=O(e^{-Ch^{-2\epsilon}}).
	\end{equation}
	We can then calculate for $y=0$ and $|x|<h^{\frac{1}{2}-\epsilon}$ that
	\begin{align}
	F'_0(x)&=\left(\frac{1}{\sqrt{2ah\pi}}\frac{1}{1+\frac{a'x}{2a}}\right) \\
	&\hspace{5mm}\times\exp\left(-\frac{1}{2}\left(\frac{2\sqrt{a}}{a'\sqrt{h}}\log\left(1+\frac{a'x}{2a}\right)+\sqrt{\frac{h}{a}}\left( b-\frac{a'}{4}\right)\right)^2\right) \\
	&=\frac{1}{\sqrt{2ah\pi}}\left(1-\frac{a'x }{2a}+O\left(h^{1-2\epsilon}\right)\right)e^{-\frac{x^2}{2ah}} \\
	&\hspace{5mm}\times\exp\left(-\frac{x}{\sqrt{ah}}\left(-\frac{a'x^2}{4\sqrt{a^3h}}+\sqrt{\frac{h}{a}}\left(b-\frac{a'}{4}\right)  \right)+O\left(h^{1-4\epsilon}\right)\right) \\
	&=\frac{1}{\sqrt{2ah\pi}}\left(1-\frac{a' x}{2a}+O\left(h^{1-2\epsilon}\right)\right)e^{-\frac{x^2}{2ah}} \\
	&\hspace{5mm}\times\left(1-\frac{x}{\sqrt{ah}}\left(-\frac{a'x^2}{4\sqrt{a^3h}}+\sqrt{\frac{h}{a}}\left(b-\frac{a'}{4}\right)  \right)+O\left(h^{1-6\epsilon}\right)\right) \\
	&=\frac{e^{-\frac{x^2}{2ah}}}{\sqrt{2ah\pi}}\left(1-\frac{x}{a}\left(b-\frac{a'}{4}+\frac{a'}{2}\right)+\frac{a'x^3}{4a^2h}+O(h^{1-6\epsilon})\right).
	\end{align}
	This is related to $\hat{\phi}$ by
	\begin{align}
		\PP\left(\hat{\phi}_{t+h,t}(y)<x\right)
		&=\PP(\phi_{-t,-t-h}(x)>y) \\
		&= F_{y+O\left(h^{\frac{3}{2}-3\epsilon}\right)}(x)+O\left(e^{-Ch^{-2\epsilon}}\right) \\
		&=F_y(x)+O\left(h^{\frac{3}{2}-3\epsilon}\sup_{y\in \RR} \frac{dF_y(x)}{dy} \right)+O\left(e^{-Ch^{-2\epsilon}}\right)
	\end{align}
	and on $\left[ y-h^{\frac{1}{2}-\epsilon},y+h^{\frac{1}{2}-\epsilon}\right]$ this is equal to
	\begin{align}
	F_y(x)+O(h^{1-3\epsilon}).
	\end{align}
	We use this to compute,
	\begin{align}
	&\EE(\hat{\phi}_{t+h,t}(y)) \\
	=&		y+\int_y^\infty \left(1-\PP\left(\hat{\phi}_{t+h,t}(y)<x\right)\right)dx
	-\int_{-\infty}^{y} \PP\left(\hat{\phi}_{t+h,t}(y)<x\right) dx \\
	=&	y+\int_y^{y+h^{\frac{1}{2}-\epsilon}} \left(1-\PP\left(\hat{\phi}_{t+h,t}(y)<x\right)\right)dx \\
          &-\int_{y-h^{\frac{1}{2}-\epsilon}}^{y} \PP\left(\hat{\phi}_{t+h,t}(y)<x\right) dx +O\left(e^{-Ch^{-2\epsilon}}\right) \\
	=&	y+\int_y^{y+h^{\frac{1}{2}-\epsilon}} \left(1-F_y(x)\right)dx
           -\int_{y-h^{\frac{1}{2}-\epsilon}}^{y} F_y(x) dx +O\left(h^{\frac{3}{2}-4\epsilon}\right) \\
          =&	y+h^{\frac{1}{2}-\epsilon}-\int_{y-h^{\frac{1}{2}-\epsilon}}^{y+h^{\frac{1}{2}-\epsilon}} F_y(x) dx +O\left(h^{\frac{3}{2}-4\epsilon}\right) \\
          =&	y+h^{\frac{1}{2}-\epsilon}-\int_{-h^{\frac{1}{2}-\epsilon}}^{h^{\frac{1}{2}-\epsilon}} F_0(x) dx +O\left(h^{\frac{3}{2}-4\epsilon}\right).
        \end{align}
        Integrating by parts we can get an expression in terms of $F_0'$ rather than $F_0$. That is
        \begin{align}
          &\EE(\hat{\phi}_{t+h,t}(y)) \\
          =&	y+h^{\frac{1}{2}-\epsilon}-h^{\frac{1}{2}-\epsilon}\left(F_0(h^{\frac{1}{2}-\epsilon})-F_0(-h^{\frac{1}{2}-\epsilon})\right)+\int_{-h^{\frac{1}{2}-\epsilon}}^{h^{\frac{1}{2}-\epsilon}} xF'_0(x)dx+O\left(h^{\frac{3}{2}-4\epsilon}\right) \\
          =&	y+\int_{-h^{\frac{1}{2}-\epsilon}}^{h^{\frac{1}{2}-\epsilon}} xF'_0(x)dx+O\left(h^{\frac{3}{2}-4\epsilon}\right).
        \end{align}
        Substituting the approximation of $F_0'(x)$ into that integral and integrating gives
        \begin{align}
          &\EE(\hat{\phi}_{t+h,t}(y)) \\
          =&y+\int_{-h^{\frac{1}{2}-\epsilon}}^{h^{\frac{1}{2}-\epsilon}} \left(x-\frac{x^2}{a}\left(b+\frac{a'}{4}\right)+\frac{a'x^4}{4a^2h}\right)\frac{e^{-\frac{x^2}{2ah}}}{\sqrt{2ah\pi}}dx+O\left(h^{\frac{3}{2}-7\epsilon}\right) \\
          =&y+\int_{-\infty}^{\infty} \left(x-\frac{x^2}{a}\left(b+\frac{a'}{4}\right)+\frac{a'x^4}{4a^2h}\right)\frac{e^{-\frac{x^2}{2ah}}}{\sqrt{2ah\pi}}dx+O\left(h^{\frac{3}{2}-7\epsilon}\right) \\
          =&y-h\left(b+\frac{a'}{4}\right)+\frac{3ha'}{4}+O\left(h^{\frac{3}{2}-7\epsilon}\right) \\
          =&y+h\left(-b+\frac{a'}{2}\right)+O(h^{\frac{3}{2}-7\epsilon}).
        \end{align}
	Similarly, we find that
	\begin{equation}
	\Var\left(\hat{\phi}_{t+h,t}(y)\right)=ah+O(h^{2-8\epsilon}) .
	\end{equation}

	Thus, the single point motions are diffusion processes with the required drift and diffusivity.
	
	Next, we will show that the motions started from $y_1$ and $y_2$ have zero covariation until they coalesce, and thus, are independent until they coalesce. This follows immediately from the fact that for $y_1 \neq y_2$
	\begin{equation}
	\mathrm{Cov}\left(\hat{\phi}_{t+h,t}(y_1),\hat{\phi}_{t+h,t}(y_2)\right)=o(h).
	\end{equation}
	To establish this fact consider the events
	\begin{equation}
	\mathcal{A}_i=\left\lbrace \sup_{0<\delta t < h}\left|\hat{\phi}_{t+\delta t,t}(y_i)-y_i\right|<\frac{|y_2-y_1|}{2} \right\rbrace \hspace{10mm} \textrm{for } i=1,2.
	\end{equation}
	On the intersection of these events, we know that the $\hat{\phi}_{t+h,t}(y_i)$ are independent as the forward flows on $[-t-h,t]\times \left[y_i-\frac{|y_2-y_1|}{2},y_i+\frac{|y_2-y_1|}{2}\right]$ are independent, and each determines the corresponding $\mathcal{A}_i$ and $\hat{\phi}_{t+h,t}(y_i)$. Thus, writing $B$ for the complement of $\mathcal{A}_1 \cap \mathcal{A}_2$,
	\begin{align}
	\left|\mathrm{Cov}\left(\hat{\phi}_{t+h,t}(y_1),\hat{\phi}_{t+h,t}(y_2)\right)\right|&=\left|\mathrm{Cov}\left(\mathbb{1}_B\hat{\phi}_{t+h,t}(y_1),\mathbb{1}_B\hat{\phi}_{t+h,t}(y_2)\right)\right| \\
                                                                                             &\leq \sqrt{\Var\left(\mathbb{1}_B\hat{\phi}_{t+h,t}(y_2)\right)\Var\left(\mathbb{1}_B\hat{\phi}_{t+h,t}(y_1)\right)}.
        \end{align}
        As a geometric mean of positive values can't be larger than the largest value this is bounded by
        \begin{align}
	& \max_{i=1,2} \left\lbrace \Var\left(\mathbb{1}_B\hat{\phi}_{t+h,t}(y_i)\right)\right\rbrace \\
	\leq& \max_{i=1,2} \left\lbrace \EE\left(\mathbb{1}_B\left(\hat{\phi}_{t+h,t}(y_i)-y_i\right)^2\right)\right\rbrace
	\end{align}
	but, as we know that
	\begin{equation}
	\PP\left(\left|\hat{\phi}_{t+h,t}(y_i)-y_i\right|>x\right)\leq 2\left(1-\Phi\left( \frac{x-b^*h}{\sqrt{a^*h}} \right)\right)
	\end{equation}
	we can deduce that
	\begin{align}
	\left|\mathrm{Cov}\left(\hat{\phi}_{t+h,t}(y_1),\hat{\phi}_{t+h,t}(y_2)\right)\right|
          & \leq \int_{\frac{|y_2-y_1|}{2}}^\infty x^2 \frac{2}{\sqrt{a^*h}}\Phi'\left( \frac{x-b^*h}{\sqrt{a^*h}} \right) dx\\
          & = \int_{\frac{|y_2-y_1|-2b^*h}{2\sqrt{a^*h}}}^\infty x^22\Phi'(u) du\\
          &=\sqrt{\frac{2}{a^*h\pi}}\int_{\frac{|y_2-y_1|-2b^*h}{2\sqrt{a^*h}}}^\infty a^*h(u+b^*h)^2 e^{-\frac{u^2}{2}} du \\
          &=\sqrt{\frac{2a^*h}{\pi}}(1+O(h))\int_{\frac{|y_2-y_1|-2b^*h}{2\sqrt{a^*h}}}^\infty u^2 e^{-\frac{u^2}{2}} du.
        \end{align}
        For sufficiently small $h$ this is bounded by
        \begin{align}
          &\sqrt{\frac{3a^*h}{\pi}}\int_{\frac{|y_2-y_1|}{3\sqrt{a^*h}}}^\infty u^2 e^{-\frac{u^2}{2}} du \\
          =&\sqrt{\frac{3a^*h}{\pi}}\left(\frac{|y_2-y_1|}{3\sqrt{a^*h}}e^{-\frac{|y_2-y_1|^2}{18a^*h}} +\int_{\frac{|y_2-y_1|}{3\sqrt{a^*h}}}^\infty e^{-\frac{u^2}{2}} du\right) \\
          =&O\left(e^{-\frac{|y_2-y_1|^2}{18a^*h}}\right).
	\end{align}
	This establishes the result.
\end{proof}

Finally, we relax the restriction that $a$ and $b$ are Lipschitz in time.

\begin{thm:reversal}
	If $a$ has spatial derivative $a'$ and $a$, $b$ and $a'$ are uniformly bounded on compacts in time and $L$-Lipschitz in space then
	\begin{equation}
	\hat\mu_A=\nu_A:=\mu^{a^\nu,b^\nu}_A
	\end{equation}
	where 	$ a^\nu(t,x)=a(-t,x)$, $b^\nu(t,x)=-b(-t,x)+a'(-t,x)/2$ and $\hat{\mu}_A$ is
	the time reversal of $\mu_A$.
\end{thm:reversal}
\begin{proof}
	 Define approximations $a_n$ and $b_n$ by
	\begin{equation}
	a_n=a\ast K_n \textrm{ and } b_n=b\ast K_n
	\end{equation}
	where $\ast$ denotes convolution in time,
	\begin{equation}
	K_n(t)=nK(tn)
	\end{equation}
	and $K$ is a smooth, non-negative function supported on $[-1,1]$ with supremum and integral equal to one. The resulting $a_n$ and $b_n$ are smooth. Thus, we will be able to apply Theorem \ref{reversalLipschitz} to a flow with these parameters.
	
	Let $\phi^n \in C^\circ (\RR,\DDD)$ be the coalescing diffusive flow driven by $a_n$ and $b_n$ and let
	\begin{equation}
	b^*_k=\sup_{[-k-1,k+1]\times [0,1]} \lvert b(t,x) \rvert
	\end{equation}
	and
	\begin{equation}
	a^*_k=\sup_{[-k-1,k+1]\times [0,1]} a(t,x).
	\end{equation}
	We define $A_N$ to be the subset of $\phi \in C^\circ (\RR,\DDD)$ such that for all $k$ both
	\begin{align}
	\left| \phi_{ts}(x)-x \right|\leq 4b^*_kk+kN\sqrt{8a^*_k}+1\hspace{5mm} \forall x\in [0,1] \hspace{5mm} \forall s,t \in [-k,k] \text{ with } s<t
	\end{align}
	and
	\begin{align}
	\left| \phi_{ts}(x)-x \right|\leq \frac{1}{k}\hspace{5mm} \forall x\in [0,1] \hspace{5mm} \forall s,t \in [-k,k] \text{ with } t-s \in [0,\delta_{k,N}]
	\end{align}
	where
	\begin{equation}
	\delta_{k,N}=\min \left\lbrace \frac{1}{18k^3 N a^*_k (1+a^*_k+b^*_k)}, \frac{a^*_k}{2{b^*_k}^2} \right\rbrace	.
	\end{equation}
	
	In Proposition \ref{ANcompact}, we prove that $A_N$ is compact; and in Proposition \ref{ANhighprob}, we prove that $\phi^n \in A_N$ with high probability in $N$ uniformly in $n$. Thus, we can deduce that the $\phi^n$ are tight. Let $\phi$ be a weak sub-sequential limit of $\phi^n$. We will show that $\phi\sim\mu_A$ and that $\hat{\phi}\sim\nu_A$, which establishes the theorem.
	
	We present here only the proof that $\phi\sim\mu_A$. The proof that $\lim_{n\rightarrow \infty}\hat\phi^n\sim\nu_A$ is identical, but considering $\hat{\phi}^n$ and $-b+\frac{a'}{2}$ instead of $\phi^n$ and $b$, it then follows that $\hat\phi\sim\nu_A$ as time reversal is an isometry.
	By Theorem \ref{webexists} it suffices to show that
	\begin{equation}
	\EE\left( \phi_{ts}(x)-\int_{s}^t b(r,\phi_{rs}(x)) dr\middle| \FFF_s \right)=x \hspace{5mm} \forall x\in [0,1] \hspace{5mm} \forall s<t
	\end{equation}
	and
	\begin{equation}
	\EE \left( M_{st}(x_1,x_2,b,a,\phi)\mid \FFF_s \right)=x_1 x_2
	\hspace{5mm} \forall x_1,x_2\in [0,1] \hspace{5mm} \forall s<t
	\end{equation}
	where
	\begin{align}
	M_{st}(x_1,x_2,b,a,\phi)=&\phi_{ts}(x_1)\phi_{ts}(x_2)-\int_{s}^t \big(\phi_{rs}(x_1)b(r,\phi_{rs}(x_2))\\&+\phi_{rs}(x_2)b(r,\phi_{rs}(x_1))\big)dr-\int_{T^{(s,x_1)(s,x_2)} \wedge t}^t a(r,\phi_{rs}(x_1)) dr.
	\end{align}
	The proof of these two statements are very similar, so we will only provide the more complicated second one here. Furthermore as $M_{st}(x_1,x_2,b,a,\phi)$ is independent of $\FFF_s$ it suffices to show that
        \begin{equation}
	\EE \left( M_{st}(x_1,x_2,b,a,\phi) \right)=x_1 x_2
	\hspace{5mm} \forall x_1,x_2\in [0,1] \hspace{5mm} \forall s<t.
	\end{equation}
	
	Proposition \ref{EMcontinuousinx} says that $\EE \left( M_{st}(x_1,x_2,b,a,\phi) \right)$ is a continuous function of $x_1$ and $x_2$. Thus, it suffices to show that for any pair of intervals $I_1$ and $I_2$,
	\begin{align}
	\EE_x\EE_\phi \left( M_{st}(x_1,x_2,b,a,\phi) \right)=\EE_x(x_1x_2)
	\end{align}
	where $\EE_x$ averages over values of $x_1$ and $x_2$ in $I_1$ and $I_2$ respectively, and $\EE_\phi$ is the same as $\EE$ on previous lines. Proposition \ref{EMcontinuousinphi} says
	\begin{equation}
	\EE_\phi\EE_x(M_{st}(x_1,x_2,b,a,\phi))=\lim_n \EE_{\phi^n}\EE_x\left( M_{st}(x_1,x_2,b,a,\phi^n) \right)
	\end{equation}
	which is used in the calculation below.
	Writing $D_n$ for $M_{st}(x_1,x_2,b,a,\phi^n)-M_{st}(x_1,x_2,b_n,a_n,\phi^n)$ we can calculate, using Proposition \ref{reversalLipschitz} in the fourth equality, that
	
	\begin{align}
	&\EE_x\EE_\phi(M_{st}(x_1,x_2,b,a,\phi)) \\
	=&\EE_\phi\EE_x(M_{st}(x_1,x_2,b,a,\phi)) \\
	=&\lim_n \EE_{\phi^n}\EE_x\left( M_{st}(x_1,x_2,b,a,\phi^n) \right)\\
	=&\lim_n \EE_{\phi^n}\EE_x\left( M_{st}(x_1,x_2,b_n,a_n,\phi^n) \right)+\lim_n \EE_{\phi^n}\EE_x\left( D_n \right) \\
	=&\EE_x\left(x_1 x_2 \right)+\lim_n \EE_{x}\EE_{\phi^n}\left( D_n \right). \\
	\end{align}
	It remains only to show that $\EE_{\phi^n}(D_n)$ goes to $0$ uniformly in $x$ as $n\rightarrow \infty$.
	
	\begin{align}
	D_n=&\int_{s}^{t}\phi^n_{rs}(x_1) (b_n(r,\phi^n_{rs}(x_2))-b(r,\phi^n_{rs}(x_2))) dr \\
	&+\int_{s}^{t}\phi^n_{rs}(x_2) (b_n(r,\phi^n_{rs}(x_1))-b(r,\phi^n_{rs}(x_1))) dr \\
	&+\int_{T^{(s,x_1)(s,x_2)}}^{t} \left(a_n(r,\phi^n_{rs}(x_1))-a(r,\phi^n_{rs}(x_1))\right) dr
	\end{align}
	
	Each of these terms has expectation tending to $0$. We will prove this for the first term (the second term is very similar and the third term is even simpler, so the same argument works). We firstly rearrange each half of the first term separately. We assume here for simplicity that $t-s>2/n$, obviously this is fine for all sufficiently large $n$.
	
	\begin{align}
	&\int_{s}^{t}\phi^n_{rs}(x_1) b_n(r,\phi^n_{rs}(x_2))dr \\
	=&\int_{s}^{t}\int_{-\frac{1}{n}}^{\frac{1}{n}}\phi^n_{rs}(x_1) b(r+u,\phi^n_{rs}(x_2))K_n(u) du dr \\
	=&\int_{s-\frac{1}{n}}^{t+\frac{1}{n}}\int_{(v-t)\vee -\frac{1}{n}}^{(v-s)\wedge \frac{1}{n}} \phi^n_{v-u,s}(x_1) b(v,\phi^n_{v-u,s}(x_2))K_n(u) du dv \\
	=&\int_{s-\frac{1}{n}}^{s+\frac{1}{n}}\int_{-\frac{1}{n}}^{(v-s)}I_1 du dv + \int_{t-\frac{1}{n}}^{t+\frac{1}{n}}\int_{v-t}^{\frac{1}{n}} I_1 du dv + \int_{s+\frac{1}{n}}^{t-\frac{1}{n}}\int_{-\frac{1}{n}}^{\frac{1}{n}} I_1 du dv
	\end{align}
	where $v=r+u$ and $I_1=\phi^n_{v-u,s}(x_1) b(v,\phi^n_{v-u,s}(x_2))K_n(u)$. The first two of these integrals are over an area that is $O(n^{-2})$ and the integrand $I_1=O(n)$, so only the final integral will contribute to the limit.
	
	\begin{align}
	&\int_{s}^{t} \phi^n_{rs}(x_1) b(r,\phi^n_{rs}(x_2)) dr \\
	=&\int_{s}^{t} \int_{-\frac{1}{n}}^{\frac{1}{n}}\phi^n_{rs}(x_1) b(r,\phi^n_{rs}(x_2)) K_n(u) du dr \\
	=&\int_{s-\frac{1}{n}}^{s+\frac{1}{n}} \int_{-\frac{1}{n}}^{\frac{1}{n}} I_2 du dr +\int_{t-\frac{1}{n}}^{t+\frac{1}{n}} \int_{-\frac{1}{n}}^{\frac{1}{n}} I_2 du dr +\int_{s+\frac{1}{n}}^{t-\frac{1}{n}} \int_{-\frac{1}{n}}^{\frac{1}{n}} I_2 du dr
	\end{align}
	where $I_2=\phi^n_{rs}(x_1) b(r,\phi^n_{rs}(x_2)) K_n(u)$. Again, the first two terms are $O(n^{-1})$, so only the last term will contribute to the limit. Combining these 2 rearrangements together and discarding small terms we find that
	\begin{align}
	&\lim_n \EE_{\phi^n}(D_n) \\=&\lim_n \EE_{\phi^n}\left(\int_{s+\frac{1}{n}}^{t-\frac{1}{n}}\int_{-\frac{1}{n}}^{\frac{1}{n}} I_1 du dv-\int_{s+\frac{1}{n}}^{t-\frac{1}{n}} \int_{-\frac{1}{n}}^{\frac{1}{n}} I_2 du dr \right) \\
	=&\lim_n \EE_{\phi^n}\int_{s+\frac{1}{n}}^{t-\frac{1}{n}} \int_{-\frac{1}{n}}^{\frac{1}{n}} I_3 K_n(u)du dr \\
	\leq&\lim_n \int_{s+\frac{1}{n}}^{t-\frac{1}{n}} \int_{-\frac{1}{n}}^{\frac{1}{n}} K_n(u)du dr \sup_{\substack{u\in[-\frac{1}{n},\frac{1}{n}] \\ r\in \left[s+\frac{1}{n},t-\frac{1}{n}\right]}} \EE_{\phi^n}I_3 \\
	\leq&(t-s)\lim_n \sup_{\substack{u\in[-\frac{1}{n},\frac{1}{n}] \\ r\in \left[s+\frac{1}{n},t-\frac{1}{n}\right]}} \EE_{\phi^n}I_3
	\end{align}
	where
	\begin{equation}
	I_3=\left(\phi^n_{r-u,s}(x_1) b(r,\phi^n_{r-u,s}(x_2))-\phi^n_{rs}(x_1) b(r,\phi^n_{rs}(x_2))\right).
      \end{equation}
      As $b$ is Lipschitz in space and the $\phi^n$ have bounded diffusivity, this final supremum convereges to zero.
\end{proof}

\vspace{10pt}
\section{Appendix}
\label{Appendix}
\vspace{10pt}

The following result is required to prove the existence of the coalescing diffusive flows, as stated in Theorem \ref{webexists}. It is a generalization of Proposition A.10 of \cite{NT} and has a similar proof.
\begin{prop}
	\label{technicalproposition}
	Let $E$ be a countable subset of $\RR^2$ containing $\QQ^2$, and let $a,b$ be measurable and uniformly bounded on compacts in time and $L$-Lipschitz in space. Then, taking $C^\circ_E=C^{\circ,+}_E\cap C^{\circ,-}_E$, we have $\mu^{a,b}_E(C^\circ_E)=1$.
\end{prop}

\begin{proof}
	Following the proof of Proposition 8.10 in \cite{NT} we will verify that each of five conditions hold a.s$\ldotp$, and as they characterize $C^\circ_E$ inside $C_E$ \cite{NT}, the result follows. Let $z$ be drawn from the distribution $\mu^{a,b}_E$ and, for $e\in E$ let $z^{e}$ denote the path starting from $e$.  The first condition is that
	\begin{equation}
	z^{(s,x+n)}_t=z^{(s,x)}_t+n, \hspace{6mm} s,t,x \in \QQ, \hspace{6mm} s<t, \hspace{6mm} n \in \ZZ.
	\end{equation}
	
	Taking $e=(s,x)$ and $e'=(s,x+n)$, we have that $T^{ee'}=s$. So by the proof of Proposition \ref{countableweb}, this condition is satisfied.
	
	Next we consider the 3 conditions
	\begin{equation}
	z^{(s,x)}_t = \inf_{y \in \QQ, y>x} z^{(s,y)}_t, \hspace{6mm} (s,x)\in E, \hspace{6mm} t \in \QQ, \hspace{6mm} t>s,
	\end{equation}
	\begin{equation}
	z^{(s,x)}_t = \sup_{y \in \QQ, y<x} z^{(s,y)}_t, \hspace{6mm} (s,x)\in E, \hspace{6mm} t \in \QQ, \hspace{6mm} t>s
	\end{equation}
	and
	\begin{equation}
	\Phi^-_{(t,u]}\circ\Phi^-_{(s,t]}\leq \Phi^-_{(s,u]}\leq \Phi^+_{(s,u]}\leq \Phi^+_{(t,u]}\circ \Phi^+_{(s,t]}, \hspace{6mm} s,t,u \in \QQ, \hspace{6mm} s<t<u.
	\end{equation}
	Where we define
	\begin{equation}
	\Phi^-_{(s,t]}(x)=\sup_{y \in \QQ, y<x} z^{(s,y)}_t, \hspace{6mm} \Phi^+_{(s,t]}(x)=\inf_{y \in \QQ, y>x} z^{(s,y)}_t.
	\end{equation}
	Let $(s,x) \in E$ and $t,u \in \QQ$, with $s\leq t<u$. Consider the event
	\begin{equation}
	A=\left\{ \sup_{y \in \QQ, y<Z^{(s,x)}_t} Z^{(t,y)}_u = Z^{(s,x)}_u = \inf_{y' \in \QQ, y'>Z^{(s,x)}_t} Z^{(t,y')}_u \right\}.
	\end{equation}
	Note that on the countable intersection, over $s,x,t,u$, of the events $A$, the above 3 conditions hold. So to show they hold a.s$\ldotp$, it suffices to show $\PP(A)=1$.
	Fix $n \in \NN$ and set $Y=n^{-1}\floor{nZ^{(s,x)}_t}$ and $Y'=Y+1/n$. Then $Y$ and $Y'$ are $\QQ$ valued, $\FFF_t$-measurable random variables. Now note that $\PP(Y<Z^{(s,x)}_t<Y')=1$ and
	\begin{equation}
	\{Y<Z^{(s,x)}_t<Y'\} \cap \{ T^{(t,Y)(t,Y')}\leq u \} \subseteq A.
	\end{equation}
	Consider the process
	\begin{equation}
	Z^{(t,Y')}_r-Z^{(t,Y)}_r-2(r-t)b^*
	\end{equation}
	as a function of $\tau$ where
	\begin{equation}
	\tau=\int_t^r \left(a(\rho, Z^{(t,Y')}_\rho)+a(\rho, Z^{(t,Y)}_\rho)\right) d\rho
	\end{equation}
	is defined to make the diffusivity of this process 1.
	
	This can be bounded above by a Brownian motion $B_\tau$ started at $1/n$. For $n$ sufficiently large that $u-t>1/n$ and
	\begin{align}
	\PP( T^{(t,Y)(t,Y')}\leq u )
	&\geq \PP \left(\inf_{\tau\leq \frac{1}{n}} B_\tau <-\frac{b^*}{na_*} \right) \\
	&= 2\Phi\left( \frac{1+b^*/a_*}{\sqrt{n}} \right) \rightarrow 1.
	\end{align}
	So $\PP(A)=1$ and the conditions hold.
	
	The final condition is that for all $\epsilon>0$ and all $n\in \NN$, there exists $\delta>0$ such that
	\begin{equation}
	\norm{\Phi_{(s,t]}-\textrm{id}}_\infty<\epsilon
	\end{equation}
	for all $s,t \in \QQ \cap (-n,n)$ with $0<t-s<\delta$.
	
	Define for $\delta>0$ and $e=(s,x)\in E$,
	\begin{equation}
	V^e(\delta)=\sup_{s\leq t \leq s+\delta^2}\abs{Z^e_t-x}.
	\end{equation}
	Then, letting $B$ be a standard Brownian motion, for sufficiently small $\delta$ and large $n$
	\begin{align}
	\PP(V^e(\delta)>n\delta)
	&\leq 2\PP\left(\sup_{s\leq t \leq s+\delta^2} B_t-B_s>\frac{n\delta-b^*\delta^2}{a^*}\right) \\
	&\leq e^{-\frac{(n-1)^2}{2{a^*}^2}}.
	\end{align}
	Consider, for each $n \in \NN$ the set
	\begin{equation}
	E_n = \left\{(j2^{-2n},k2^{-n}): j \in \frac{1}{2}\ZZ \cap [-n^{\frac{1}{3}}2^{2n},n^{\frac{1}{3}}2^{2n}),k=0,1,\dots,2^n-1  \right\}
	\end{equation}
	and the event
	\begin{equation}
	A_n = \bigcup_{e\in E_n} \{V^e(2^{-n})>n2^{-n}\}.
	\end{equation}
	Then, $\PP(A_n)\leq \abs{E_n}\sup_{e\in E_n}\PP\left(V^e(2^{-n})>n2^{-n}\right)$ and

        \begin{align}
          \PP\left(V^e(2^{-n})>n2^{-n}\right)\leq& \PP\left(\sup_{s\leq t \leq s+2^{-2n}} Z^e_t-x>n2^{-n}\right) \\
                                                 &+\PP\left(\inf_{s\leq t \leq s+2^{-2n}} Z^e_t-x<n2^{-n}\right).
        \end{align}
        For $n$ sufficiently large that $2^n>b^*$, both of these terms are $O\left(e^{-\frac{(n-1)^2}{2}}\right)$ by the reflection principle. As $\abs{E_n}=e^{O(n)}$, we can conclude that $\sum_n \PP(A_n) < \infty$, so by Borel-Cantelli, almost surely there exists some $N<\infty$ such that $V^e(2^{-n})\leq n2^{-n}$ for all $e \in E_n$, for all $n\geq N$.
	
	Given $\epsilon >0$, choose $n\geq N$ such that $(4n+2)2^{-n} \leq \epsilon$ and set $\delta = 2^{-2n-1}$. Then, for all rationals $s,t \in (-n,n)$ with $0<t-s<\delta$ and all rationals $x\in [0,1]$, there exist $e^\pm = (r,y^\pm) \in E_n$ such that
	\begin{align}
	r \leq s<t\leq r+2^{-2n}, \\
	x+n2^{-n}<y^+\leq x+(n+1)2^{-n}, \\
	x-(n+1)2^{-n}\leq y^-<x-n2^{-n},
	\end{align}
	then, $Z^{e^-}_s<x<Z^{e^+}_s$, so
	\begin{equation}
	x-\epsilon \leq Z^{e^-}_t\leq Z^{(s,x)}_t \leq Z^{e^+}_t \leq x+\epsilon.
	\end{equation}
	Hence, the final condition holds almost surely and thus the proposition holds.
\end{proof}

The rest of the propositions in this appendix are used in the direct proof of Theorem \ref{reversalidentify} in Section \ref{DirectProof}. The definition of $A_N$ can be found in that proof.

\begin{prop}
	\label{ANcompact}
	$A_N$ is compact
\end{prop}
\begin{proof}
	$A_N$ is a closed subset of $C^\circ(\RR,\DDD)$, and so is complete. Therefore, by a diagonal argument, it suffices to show that for all $\epsilon>0$ and for all sequences $\mathcal{S}$ in $A_N$, there exists a subsequence $\mathcal{S}'$ that is contained in a ball of radius $\epsilon$.
	
	To this end take $M$ such that
	\begin{equation}
	\sum_{m=M+1}^{\infty}2^{-m}<\frac{\epsilon}{2}
	\end{equation}
	then we have that
	\begin{equation}
	d_C(\phi,\psi)<\sum_{m=1}^{M}2^{-m}d^{(m)}_C(\phi,\psi)+\frac{\epsilon}{2} \hspace{5mm} \forall \phi,\psi \in C^\circ(\RR,\DDD).
	\end{equation}
	Thus, it suffices to find a subsequence $\mathcal{S}'$ where, for $m=1$ to $M$, we have
	\begin{equation}
	\label{dmcbound}
	d^{(m)}_C(\phi,\psi)<\frac{\epsilon}{2}\hspace{5mm} \forall \phi,\psi \in \mathcal{S}'.
	\end{equation}
	As $d^{(m)}_C$ is increasing in $m$, it suffices for this to hold for $m=M$. Note that $d^{(M)}_C$ only depends on the flows between times in $[-M,M]$. By the definition of $A_N$, the set of paths from a given point, for each of the flows in $\mathcal{S}$, is uniformly bounded and equicontinuous when restricted to the interval $[-M,M]$. This interval is also compact, so by the Arzel\`{a}-Ascoli Theorem, the set of such restricted paths is compact in the uniform norm. Using this compactness we can, for a finite set $E_{\epsilon,M,N}\subset[-M,M]\times [0,1]$, find a subsequence $\mathcal{S}'$ of $\mathcal{S}$ such that
	\begin{equation}
	\label{unifbound}
	\lVert \phi_{\cdot s}(x)-\psi_{\cdot s}(x) \rVert_{L^\infty([s,M])}<\frac{\epsilon}{2} \hspace{5mm} \forall \phi,\psi \in \mathcal{S}' \hspace{5mm} \forall (s,x)\in E_{\epsilon,M,N}.
	\end{equation}
	Let $[i]=\{1,\dots,i\}$. We will take the $\mathcal{S}'$ corresponding to
	\begin{align}
	E_{\epsilon,M,N}=\left\lbrace \left(-M+m\delta_{K,N} ,\frac{l\epsilon}{6}\right) : m\in \left[\left\lceil\frac{2M}{\delta_{K,N}}\right\rceil\right], l\in \left[\left\lceil \frac{6}{\epsilon}\right\rceil \right] \right\rbrace
	\end{align}
	where $K=\max\left\{ \ceil{\frac{6}{\epsilon}},M \right\}$.
	It remains to show from \eqref{unifbound} that \eqref{dmcbound} holds for $m=M$, i.e.
	\begin{equation}
	\label{distanceBound}
	d_\DDD(\phi_{ts},\psi_{ts})<\frac{\epsilon}{2} \hspace{5mm} \forall s,t\in[-M,M],s<t \hspace{5mm} \forall \phi,\psi \in \mathcal{S}'.
	\end{equation}
	By the definition of $d_\DDD$ this is the same as saying that for all $s,t,\phi,\psi$ and all $x$
	\begin{equation}
	\label{finalthing}
	\psi_{ts}\left(x-\frac{\epsilon}{2}\right)< \phi_{ts}(x)+\frac{\epsilon}{2}
	\end{equation}
	and
	\begin{equation}
	\phi_{ts}\left(x-\frac{\epsilon}{2}\right)< \psi_{ts}(x)+\frac{\epsilon}{2}.
	\end{equation}
	We will show the first of these the other follows by symmetry.
	
	Given $s,t,\phi,\psi$ as in \eqref{distanceBound}, there exists
	\begin{equation}
	(u,y)\in [s,s+\delta_{K,N}]\times \left(x-\frac{\epsilon}{3},x-\frac{\epsilon}{6}\right)\cap E_{\epsilon,M,N}
	\end{equation}
	and by the equicontinuity condition in the definition of $A_N$
	\begin{equation}
	\psi_{us}\left(x-\frac{\epsilon}{2}\right)<y
	\end{equation}
	\begin{equation}
	\phi_{us}(x)>y.
	\end{equation}
	Putting these together with \eqref{unifbound} we get
	\begin{equation}
	\psi_{ts}(x-\frac{\epsilon}{2})\leq \psi_{tu}(y)< \phi_{tu}(y)+\frac{\epsilon}{2}\leq \phi_{ts}(x)+\frac{\epsilon}{2}.
	\end{equation}
	This is Equation \eqref{finalthing} and so we are done.
\end{proof}

\begin{prop}
	\label{ANhighprob}
	As $N\rightarrow \infty$
	\begin{equation}
	\PP(\phi^n\in A_N)\rightarrow 1
	\end{equation}
	uniformly in $n$.
\end{prop}
\begin{proof}
	Throughout $W_t$ is a standard Brownian motion. We start by showing that w.h.p$\ldotp$ the condition that gives uniform boundedness on compact intervals holds.
	\begin{align}
	&\PP\left(\lvert\phi^n_{ts}(x)-x\rvert<4b^*_kk+kN\sqrt{8a^*_k}+1 \hspace{5mm} \forall x\in[0,1] \hspace{5mm} \forall s<t \in [-k,k]\right) \\
	\geq&\PP\left(\sup_{t\in [-k,k]} \lvert \phi^n_{t,-k}(0) \rvert<2b^*_kk+kN\sqrt{2a^*_k}\right) \\
	\geq&1-4\PP\left(\sqrt{a^*_k}W_{2k}>kN\sqrt{2a^*_k}\right) \\
	=&1-4\Phi\left(-N\sqrt{k}  \right)
	\end{align}
	and thus
	\begin{align}
	&\PP\left(\lvert\phi^n_{ts}(x)-x\rvert<4b^*_kk+kN\sqrt{8a^*_k}+1 \hspace{4mm} \forall s<t \in [-k,k] \hspace{4mm} \forall x\in [0,1] \hspace{4mm} \forall k\right) \\
	\geq& 1-4\sum_{k=1}^{\infty} \Phi(-N\sqrt{k})\rightarrow 1.
	\end{align}
	Now we will show that w.h.p$\ldotp$ the equicontinuity requirement on compact intervals holds. Let
	\begin{align}
	E_{k,N}=\left\lbrace \left(k-m\delta_{k,N} ,\frac{l}{3k}\right) : m\in \left\lbrace 1,\dots ,\left\lceil\frac{2k}{\delta_{k,N}}\right\rceil\right\rbrace, l\in \{ 1,\dots, 3k \} \right\rbrace.
	\end{align}
	The below calculation says that with high probability for all $k$ paths from each of these points will not move more than $\frac{1}{3k}$ from their stating point within time $2\delta_{k,N}$ and the non-crossing property then implies the required equicontinuity. It proceeds as follows,
	\begin{align}
	&\PP\left( \left| \phi^n_{ts}(x)-x \right|\leq \frac{1}{k}\hspace{5mm} \forall x\in [0,1] \hspace{5mm} \forall s,t \in [-k,k] \text{ with } t-s \in [0,\delta_{k,N}]  \right) \\
	\geq&\PP\left( \sup_{t\in [s,s+2\delta_{k,N}]} \left| \phi^n_{ts}(x)-x \right| <\frac{1}{3k} \hspace{5mm} \forall (s,x)\in E_{k,N}  \right) \\
	\geq&1-4\lvert E_{k,N} \rvert \PP \left( \sqrt{a^*_k}W_{2\delta_{k,N}}+2\delta_{k,N} b^*_k >\frac{1}{3k} \right) \\
	=&1-12k\left\lceil \frac{2k}{\delta_{k,N}} \right\rceil \Phi\left( -\frac{1}{\sqrt{2\delta_{k,N}a^*_k}} \left(\frac{1}{3k}-2\delta_{k,N}b^*_k \right) \right) \\
	\geq& 1-\frac{36k^2}{\delta_{k,N}} \Phi\left( -\frac{1}{\sqrt{18k^2a^*_k\delta_{k,N}}}+\sqrt{\frac{2{b^*_K}^2\delta_{k,N}}{a^*_k}} \right) \\
	\geq&1-\max \left\lbrace \frac{72k^2{b^*_k}^2}{a^*_k} , 648k^5Na^*_k(1+a^*_k+b^*_k) \right\rbrace \Phi\left( -\sqrt{kN(1+a^*_k+b^*_k)}+1\right).
	\end{align}
	As the maximum can be bounded by a polynomial in $k,N,a^*_k$ and $b^*_k$, and $\Phi(\dots)$ is decreasing exponentially in all of those variables, we can conclude by use of a union bound that
	\begin{equation}
	\PP\left( \left| \phi^n_{ts}(x)-x \right|\leq \frac{1}{k} \hspace{4mm}\forall x\in [0,1] \hspace{4mm} \forall s,t \in [-k,k] \text{ with } t-s \in [0,\delta_{k,N}] \hspace{4mm} \forall k \right)
	\end{equation}
	$\rightarrow 1$ as $N\rightarrow \infty$.
\end{proof}

\begin{prop}
	\label{EMcontinuousinx}
	$\EE \left( M_{st}(x_1,x_2,b,a,\phi) \right)$ is a continuous function of $x_1$ and $x_2$.
\end{prop}
\begin{proof}
	We will show that
	\begin{equation}
	\lvert \EE_\phi\left(M_{st}(x_1,x_2,b,a,\phi)\right)-\EE_\phi\left(M_{st}(x'_1,x'_2,b,a,\phi)\right) \rvert \rightarrow 0
	\end{equation}
	uniformly for $d_{eucl}((x_1,x_2),(x'_1,x'_2))<\delta$ as $\delta\rightarrow 0$. We start by decomposing $M_{st}(x_1,x_2,b,a,\phi)$ into the integrals up to time $s+\delta$ and the rest. The integrals up until time $s+\delta$ are
	\begin{equation}
	-\int_{s}^{s+\delta} \left(\phi_{rs}(x_1)b(r,\phi_{rs}(x_2))+\phi_{rs}(x_2)b(r,\phi_{rs}(x_1))\right)dr
	\end{equation}
	and
	\begin{equation}
	-\int_{T^{(s,x_1)(s,x_2)} \wedge t}^{(T^{(s,x_1)(s,x_2)} \wedge t) \vee (s+\delta)} a(r,\phi_{rs}(x_1)) dr.
	\end{equation}
	Taking expected value w.r.t$\ldotp$ $\phi$ and exchanging order of integration leaves two integrals with length at most $\delta$ and integrands bounded by
	\begin{equation}
	b^*\sup_{r\in [s,s+\delta]} \EE_\phi(\abs{\phi_{rs}(x_1)}+\abs{\phi_{rs}(x_2)}) \text{ and }a^*
	\end{equation}
	respectively. As $\phi_{rs}(x_i)$ is uniformly integrable for $r\leq t$ these integrals contribute only $O(\delta)$ to $M$. Thus they can be neglected.
	
	We will use $M^\delta_{st}$ to mean $M_{st}$ minus the integrals we have just shown are $O(\delta)$. Note that
	\begin{equation}
	\EE_\phi\left(M^\delta\right)=\EE_\phi\left(\EE_\phi\left(M^\delta|\FFF_{s+\delta}\right)\right)
	\end{equation}
	and by the strong Markov property
	\begin{equation}
	\EE_\phi\left(M^\delta_{st}(x_1,x_2,b,a,\phi)|\FFF_{s+\delta}\right)
	\end{equation}
	is a function of $\phi_{s+\delta,s}(x_1)$ and $\phi_{s+\delta,s}(x_2)$. Proposition \ref{mixingdiffusions} says that
	\begin{equation}
	d_{TV}((\phi_{s+\delta,s}(x_1),\phi_{s+\delta,s}(x_2)),(\phi_{s+\delta,s}(x'_1),\phi_{s+\delta,s}(x'_2)))\rightarrow 0
	\end{equation}
	so we can deduce that
	\begin{equation}
	d_{TV}\left(\EE_\phi\left(M^\delta_{st}(x_1,x_2,b,a,\phi)|\FFF_{s+\delta}\right),\EE_\phi\left(M^\delta_{st}(x'_1,x'_2,b,a,\phi)|\FFF_{s+\delta}\right)\right)\rightarrow 0.
	\end{equation}
	Combining this with the fact that $\EE_\phi(M^\delta_{st}(x_1,x_2,b,a,\phi)|\FFF_{s+\delta})$ is uniformly integrable for $(x_1,x_2)$ in each compact set, we are done.
\end{proof}

\begin{prop}
	\label{mixingdiffusions}
	\begin{equation}
	d_{TV}((\phi_{s+\delta,s}(x_1),\phi_{s+\delta,s}(x_2)),(\phi_{s+\delta,s}(x'_1),\phi_{s+\delta,s}(x'_2)))\rightarrow 0
	\end{equation}
	uniformly for $d_{eucl}((x_1,x_2),(x'_1,x'_2))< \delta$ as $\delta\rightarrow 0$.
\end{prop}
\begin{proof}
	Let $\tilde{\phi}$ have the same distribution as $\phi$ but be coupled with $\phi$ such that, for each $i=1,2$, we have $\phi_{t,s}(x_i)$ and $\tilde{\phi}_{t,s}(x_i')$ evolve independently until they take the same value at which point they coalesce. This is possible, as having fixed $\phi$ we can construct $\tilde{\phi}$ by first constructing $\tilde{\phi}(x_1')$ independently until it hits $\phi(x_1)$, then constructing $\tilde{\phi}(x_2')$ independently until it hits $\tilde{\phi}(x_1')$ or $\phi(x_2)$.
	
	\begin{align}
	&d_{TV}((\phi_{s+\delta,s}(x_1),\phi_{s+\delta,s}(x_2)),(\phi_{s+\delta,s}(x'_1),\phi_{s+\delta,s}(x'_2))) \\
	\leq& 1-\PP\left(\phi_{s+\delta,s}(x_i)=\tilde{\phi}_{s+\delta,s}(x_i) \textrm{ for both } i=1,2\right) \\
	\leq& \PP\left(\phi_{s+\delta,s}(x_1)\neq\tilde{\phi}_{s+\delta,s}(x_1)\right)+\PP\left(\phi_{s+\delta,s}(x_2)\neq\tilde{\phi}_{s+\delta,s}(x_2)\right)\\
	\leq& 2(1-2\PP(a_* W_\delta <-\delta-\delta b^*))\\
	=& 2\left(1-2\Phi\left(-\sqrt{\delta}\frac{1+b^*}{a_*}\right)\right)\\
	\leq& 2\left(1-2\left(0.5-\sqrt{\frac{\delta}{2\pi}}\frac{1+b^*}{a_*}\right)\right)\\
	=& \sqrt{\frac{8\delta}{\pi}}\frac{1+b^*}{a_*} \rightarrow 0.
	\end{align}
\end{proof}

\begin{prop}
	\label{EMcontinuousinphi}
	\begin{equation}
		\EE_\phi\EE_x(M_{st}(x_1,x_2,b,a,\phi))=\lim_n \EE_{\phi^n}\EE_x\left( M_{st}(x_1,x_2,b,a,\phi^n) \right)
	\end{equation}
\end{prop}
\begin{proof}
	We would like to be able to say that $\EE_x(M_{st})$ is a continuous function of $\phi$, and apply weak convergence. Unfortunately, even after averaging over x, this still isn't true, as $T^{(s,x_1)(s,x_2)}$ is not a continuous function of $\phi$, so we now proceed to smooth $M_{st}$ even more. Define
	\begin{equation}
	T_{\eta}=\inf\left\lbrace r\geq s :d\right(\phi_{rs}(x_1),\ZZ+\phi_{rs}(x_2)\left) <\eta \right\rbrace
	\end{equation}
	and then define
	\begin{align}
		\tilde{M}^\epsilon_{st}(x_1,x_2,b,a,\phi)=&\phi_{ts}(x_1)\phi_{ts}(x_2)-\int_{s}^t \big(\phi_{rs}(x_1)b(r,\phi_{rs}(x_2))\\&+\phi_{rs}(x_2)b(r,\phi_{rs}(x_1))\big)dr-\frac{1}{\epsilon}\int_0^\epsilon \int_{T_\eta \wedge t}^t a(r,\phi_{rs}(x_1)) dr d\eta.
	\end{align}
	By applying the triangle inequality the following three claims will now suffice to complete the proof. Firstly
	\begin{equation}
	\EE_{\phi^n}\EE_x \tilde{M}^\epsilon(\phi^n) \rightarrow \EE_{\phi}\EE_x \tilde{M}^\epsilon(\phi) \textrm{ as }n\rightarrow \infty
	\end{equation}
	secondly
	\begin{equation}
	\EE_{\phi}\EE_x \tilde{M}^\epsilon(\phi) \rightarrow \EE_{\phi}\EE_x M(\phi) \textrm{ as }\epsilon\rightarrow 0
	\end{equation}
	and thirdly
	\begin{equation}
	\EE_{\phi^n}\EE_x \tilde{M}^\epsilon(\phi^n) \rightarrow \EE_{\phi^n}\EE_x M(\phi^n) \textrm{ as }\epsilon \rightarrow 0 \textrm{ uniformly in } n.
	\end{equation}
	
	We first prove the second claim. Note that $T_\eta$ monotonically increases to $T_0$ as $\eta\rightarrow 0$ and thus $\tilde{M}^\epsilon$ is monotonically increasing to $M$ as $\epsilon\rightarrow 0$. Thus the second claim holds by the Monotone Convergence Theorem.
	
	We next prove the third claim. We have that
	\begin{align}
	\lvert \tilde{M}^\epsilon(\phi^n)-M(\phi^n) \rvert =&\frac{1}{\epsilon}\int_0^\epsilon \int_{T_\eta\wedge t}^{T_0 \wedge t} a(r,\phi^n_{rs}(x_1))dr d\eta \\
	\leq&a^*\left((T_0-T_\epsilon)\wedge (t-s)\right)
	\end{align}
	and thus
	\begin{align}
	\left\lvert\EE_{\phi^n}\left(\tilde{M}^\epsilon(\phi^n)-M(\phi^n)\right) \right\rvert\leq&a^*\left(\epsilon+(t-s)\PP(T_0-T_\epsilon>\epsilon)\right).
	\end{align}
	Using the strong Markov property at time $T_\epsilon$ we can see that
	\begin{align}
	\PP(T_0-T_\epsilon>\epsilon)\leq&1-2\PP(2a_*W_\epsilon<-\epsilon(1+2b^*)) \\
	=&1-2\Phi\left( -\frac{\sqrt{\epsilon}(1+2b^*)}{2a_*}\right)\\
	=&O(\sqrt{\epsilon}).
	\end{align}
	Putting this together and averaging over $x$, we get the third claim with the order of the expectations swapped. Note that each term of $\tilde{M}^\epsilon(\phi^n)$ and $M(\phi^n)$ have sub-exponential tails. Thus, we can apply Fubini's theorem to deduce the third claim.

	Finally, we will show that $\EE_x\tilde{M}^\epsilon(\phi)$ is a continuous function of $\phi$, from which our first claim immediately follows due to weak convergence. Then we will be done.
	
	Fix $\phi'\in C^\circ(\RR,\DDD)$. Let $\phi''$ be distance at most $\delta$ from $\phi'$. Fix $t>s$. We have, for $\delta<1$, that
	\begin{equation}
	\phi_{ts}'(x_1)-1-\delta\leq \phi_{ts}'(x_1-\delta)-\delta\leq \phi_{ts}''(x_1)\leq \phi_{ts}'(x_1+\delta)+\delta \leq \phi_{ts}'(x_1)+1+\delta.
	\end{equation}
	Let $[l_1,u_1]$ be the interval that $x_1$ is being averaged over, then
	\begin{align}
	\EE_{x_1}\left(\phi_{ts}''(x_1)\right)&\leq \frac{1}{u_1-l_1} \int_{l_1}^{u_1}\phi_{ts}'(x_1+\delta)+\delta dx_1 \\
	&\leq \frac{1}{u_1-l_1}\int_{l_1}^{u_1-\delta}\phi_{ts}'(x_1+\delta)dx_1+\frac{1}{u_1-l_1}\int_{u_1-\delta}^{u_1}\phi_{ts}'(x_1)+1dx_1 +\delta \\
	&\leq \EE_{x_1}\left(\phi_{ts}'(x_1)\right)-\frac{1}{u_1-l_1}\int_{l_1}^{l_1+\delta}\phi_{ts}'(x_1)dx_1\\
	&\hspace{77pt}+\frac{1}{u_1-l_1}\int_{u_1-\delta}^{u_1}\phi_{ts}'(x_1)dx_1+\delta\left(\frac{1}{u_1-l_1}+1\right) \\
	&\leq \EE_{x_1}\left(\phi_{ts}'(x_1)\right)+\delta\left(1+\frac{1+2\sup_{l_1\leq x_1\leq u_1}|\phi_{ts}'(x_1)|}{u_1-l_1}\right) \\
	&\rightarrow \EE_{x_1}\left(\phi_{ts}'(x_1)\right) \textrm{ as } \delta \rightarrow 0.
	\end{align}
	The lower bound is similar. We can deduce that $\EE_{x_1}(\phi_{ts}(x_1))$ is continuous in $\phi$, and so the first term of $\EE_x\tilde{M}^\epsilon$, i.e.
	\begin{equation}
          \EE_{x_1}(\phi_{ts}(x_1))\EE_{x_2}(\phi_{ts}(x_2))
	\end{equation}
	is also continuous in $\phi$. To show that the second term of $\EE_x\tilde{M}^\epsilon$ is continuous in $\phi$, as $\phi''_{rs}(x_1)$ is bounded uniformly over $r\in[s,t]$ and $\phi''$ for fixed $\delta$ and $\phi'$, it suffices to show that
	\begin{equation}
	\EE_{x_1}b(t,\phi_{ts}(x_1))
	\end{equation}
	is continuous in $\phi$.
	
	\begin{align}
	b\left(t,\phi_{ts}''(x_1)\right)&\leq \sup_{|\delta x |\leq \delta}b\left(t,\phi_{ts}'(x_1+\delta x)\right)+L\delta \\
	&\leq b\left(t,\phi_{ts}'(x_1)\right)+L\left(\delta+\sup_{|\delta x |\leq \delta} \left|\phi_{ts}'(x_1+\delta x)-\phi_{ts}'(x_1)\right|\right).
	\end{align}
	Cut the interval $[l_1-\delta,u_1+\delta]$ into pieces of length $\delta$. Let $C$ be the amount that $\phi_{ts}$ increases by over that interval. Call a piece bad if $\phi_{ts}$ increases by more than $\frac{\sqrt{\delta}}{2}$ on that piece or either of the neighbouring pieces. As $\phi_{ts}$ is non-decreasing there can be at most $\frac{6C}{\sqrt{\delta}}$ bad pieces. If $x_1$ is not in a bad piece then
	\begin{equation}
	\sup_{|\delta x |\leq \delta} \left|\phi_{ts}'(x_1+\delta x)-\phi_{ts}'(x_1)\right|\leq \sqrt{\delta}.
	\end{equation}
	So our bound on $b\left(t,\phi_{ts}''(x_1)\right)$ gives
	\begin{equation}
	b\left(t,\phi_{ts}''(x_1)\right)\leq b\left(t,\phi_{ts}'(x_1)\right)+ L\left(\delta+\sqrt{\delta}+C\mathbbm{1}_{\lbrace x_1\in \textrm{a bad piece}\rbrace}\right).
	\end{equation}
	Combining this with the corresponding lower bound whose derivation is similar we find
	\begin{align}
	&\left| \EE_{x_1}\left(b\left(t,\phi_{ts}'' (x_1)\right)\right)-\EE_{x_1}\left(b\left(t,\phi_{ts}' (x_1)\right)\right) \right|\\
	\leq& L\left(\delta+\sqrt{\delta}+C\PP_{x_1}(x_1 \in\textrm{a bad piece})\right) \\
	\leq& L\left(\delta+\sqrt{\delta}+\frac{6C^2\sqrt{\delta}}{u_1-l_1}\right)=O\left(\sqrt{\delta}\right).
	\end{align}
	Thus, $\EE_{x_1}b(t,\phi_{ts}(x_1))$  and the second term of $\EE_x \tilde{M}^\epsilon$ are continuous in $\phi$.
	
	Similarly, we can conclude that $\EE_{x_1}a(t,\phi_{ts}(x_1))$ is continuous with respect to $\phi$, and further, as the products of intervals generate the Borel $\sigma$-algebra on $\RR^2$, that
	\begin{equation}
	\label{generalSmooth}
	\int a(t,\phi_{ts}(x_1)) d\mu(x)
	\end{equation}
	is a continuous function of $\phi$ for each measure $\mu$ that is bounded, compactly supported and absolutely continuous with respect to Lebesgue measure on $\RR^2$. Note that the integral here is over $x_1$ and $x_2$, the integrand is independent of the latter but we phrase the above fact in this form as that is how it will be used later. This will be useful after we rewrite the third term of $\EE_x \tilde{M}^\epsilon$ as
	\begin{align}
	&\EE_x \frac{1}{\epsilon}\int_0^\epsilon \int_{T_\eta \wedge t}^t a(r,\phi_{rs}(x_1)) dr d\eta \\
	=&\EE_x\int_0^\epsilon \int_s^t a(r,\phi_{rs}(x_1))\frac{\mathbbm{1}_{\lbrace r>T_\eta\rbrace}}{\epsilon} dr d\eta \\
	=&\int_s^t\EE_x\left(a(r,\phi_{rs}(x_1))\int_0^\epsilon \frac{\mathbbm{1}_{\lbrace r>T_\eta\rbrace}}{\epsilon} d\eta \right) dr.
	\end{align}
	To show this is continuous it suffices to show that
	\begin{equation}
	\EE_x\left(a(t,\phi_{ts}(x_1))\int_0^\epsilon \frac{\mathbbm{1}_{\lbrace t>T_\eta\rbrace}}{\epsilon} d\eta \right)
	\end{equation}
	is continuous and uniformly bounded for all $t>s$. The boundedness is immediate. The continuity is not immediate from \eqref{generalSmooth} being continuous, because $T_\eta$ depends on $\phi$. However, it can be shown as follows. Let $T_\eta^0$ be the $T_\eta$ corresponding to $\phi'$ and define $T_\eta^\delta$ similarly. Let $\mu'$ be the measure on $\RR^2$ with Radon-Nikodym derivative
	\begin{equation}
	\int_0^\epsilon \frac{\mathbbm{1}_{\lbrace t>T_\eta\rbrace}}{\epsilon} d\eta
	\end{equation}
	with respect to the uniform probability measure on $I_1\times I_2$ and define $\mu''$ similarly. Then
	\begin{align}
	&\left|\EE_x\left(a\left(t,\phi_{ts}''(x_1)\right)\int_0^\epsilon \frac{\mathbbm{1}_{\lbrace t>T^\delta_\eta\rbrace}}{\epsilon} d\eta \right)-\EE_x\left(a\left(t,\phi_{ts}'(x_1)\right)\int_0^\epsilon \frac{\mathbbm{1}_{\lbrace t>T^0_\eta\rbrace}}{\epsilon} d\eta \right)\right| \\
	=&\left| \int a\left(t,\phi_{ts}''(x_1)\right) d\mu''(x_1)-\int a\left(t,\phi_{ts}'(x_1)\right) d\mu'(x_1)\right| \\
	\leq& \left| \int a\left(t,\phi_{ts}''(x_1)\right) d\mu''(x_1)-\int a\left(t,\phi_{ts}''(x_1)\right) d\mu'(x_1)\right| \\
	&+\left| \int a\left(t,\phi_{ts}''(x_1)\right) d\mu'(x_1)-\int a\left(t,\phi_{ts}'(x_1)\right) d\mu'(x_1)\right|.
	\end{align}
	The second of these terms is small due to the continuity of \eqref{generalSmooth}, the first term is bounded by
	\begin{equation}
	\frac{a^*}{\epsilon}\int_{l_1}^{u_1}\int_{l_2}^{u_2}\int_0^\epsilon \left|\mathbbm{1}_{\lbrace t>T_\eta^\delta \rbrace} -\mathbbm{1}_{\lbrace t>T_\eta^0 \rbrace}\right| d\eta dx_2 dx_1.
	\end{equation}
	The contribution to this integral when $|x_1-x_2|<2\delta$ is clearly small. We will show that the contribution when $x_1\geq x_2+2\delta$ is small and, as the case for $x_1\leq x_2-2\delta$ is similar, we will then be done. Conditional on $x_1\geq x_2+2\delta$ we have
	\begin{equation}
	T^0_{\eta+2\delta}(x_1-\delta,x_2+\delta)\leq T^\delta_\eta(x_1,x_2)\leq T^0_{\eta-2\delta}(x_1+\delta,x_2-\delta)
	\end{equation}
	and thus our integrand is zero unless
	\begin{equation}
        \label{tnonzero}
	t\in \left[ T^0_{\eta+2\delta}(x_1-\delta,x_2+\delta),T^0_{\eta-2\delta}(x_1+\delta,x_2-\delta)\right].
	\end{equation}
	To see that the integral is small we will change the variables of the integral. We will use an orthonormal substitution to change the variables to, $\alpha=(2\eta-x_1+x_2)/\sqrt{6}$, $\beta=(-\eta-x_1+x_2)/\sqrt{3}$, and $\gamma=(x_1+x_2)/\sqrt{2}$. For some values of the endpoints, the result is as follows
        \begin{equation}
          \frac{a^*}{\epsilon}\int_{\gamma_1}^{\gamma_2}\int_{\beta_1}^{\beta_2}\int_{\alpha_1}^{\alpha_2} \left|\mathbbm{1}_{\lbrace t>T_\eta^\delta \rbrace} -\mathbbm{1}_{\lbrace t>T_\eta^0 \rbrace}\right| d\alpha d\beta d\gamma.
	\end{equation}
        Note that the end points of the integral are independent of $\delta$. Thus the whole expression has at most the same order as the inner integral for small $\delta$. For any fixed value of $\beta$ and $\gamma$, the interval of $\alpha$ for which Equation \eqref{tnonzero} can hold has length at most $2\sqrt{6}\delta$ (as for values of $\alpha$ differeing by more than that, the corresponding intervals on the right hand side are disjoint). As the integrand is bounded by one, the inner integral is bounded by $2\sqrt{6}\delta$. Therefore, the total expression is $O(\delta)$.
\end{proof}

\section*{Acknowledgements}

I would like to thank my PhD supervisor James Norris for providing the idea for this research and to thank my anonymous reviewers and my PhD examinors Jason Miller and Amanda Turner for providing extensive help getting this manuscript into a readable condition.

\bibliographystyle{plain}
\bibliography{SecondRevision}

\end{document}